\documentclass[12pt,fleqn,draft]{article} 
\usepackage{amsmath,amssymb,amsthm,amsfonts,bm}
\usepackage{enumerate}
\usepackage{indentfirst}
\usepackage{cite}
\usepackage[usenames,dvipsnames]{color}

\makeatletter
\def\@cite#1#2{[{{\bfseries #1}\if@tempswa , #2\fi}]}
\renewcommand{\section}{%
\@startsection{section}{1}{\z@}
{0.5truecm plus -1ex minus -.2ex}%
{1.0ex plus .2ex}{\bfseries\large}}
\def\@seccntformat#1{\csname the#1\endcsname.\ }
\makeatother

\numberwithin{equation}{section} 
\theoremstyle{theorem}
\newtheorem{thm}{Theorem}[section]

\newtheorem{lem}[thm]{Lemma}

\theoremstyle{definition}
\newtheorem{df}{Definition}[section]
\newtheorem{remark}{Remark}[section]
\newtheorem{example}{Example}[section]

\newtheorem*{prth1.1}{Proof of Theorem 1.1}
\newtheorem*{prth1.2}{Proof of Theorem 1.2}
\newtheorem*{prth1.3}{Proof of Theorem 1.3}

\newcommand{\ep}{\varepsilon}

\let\hat\widehat

\def\Pi{\hat\pi}

\begin{document}
\footnote[0]
    {2010 {\it Mathematics Subject Classification}\/: 
    35A35, 47N20, 35G30, 35L70.  
    }
\footnote[0]
    {{\it Key words and phrases}\/: 
    simultaneous abstract evolution equations; 
    parabolic-hyperbolic phase-field systems; existence; time discretizations; 
    error estimates.} 
\begin{center}
    \Large{{\bf Time discretization of 
an initial value problem for 
a simultaneous abstract 
evolution equation applying to \\ 
parabolic-hyperbolic phase-field systems}}
\end{center}
\vspace{5pt}
\begin{center}
    Shunsuke Kurima%
    \\
    \vspace{2pt}
    Department of Mathematics, 
    Tokyo University of Science\\
    1-3, Kagurazaka, Shinjuku-ku, Tokyo 162-8601, Japan\\
    {\tt shunsuke.kurima@gmail.com}\\
    \vspace{2pt}
\end{center}
\begin{center}    
    \small \today
\end{center}

\vspace{2pt}
\newenvironment{summary}
{\vspace{.5\baselineskip}\begin{list}{}{%
     \setlength{\baselineskip}{0.85\baselineskip}
     \setlength{\topsep}{0pt}
     \setlength{\leftmargin}{12mm}
     \setlength{\rightmargin}{12mm}
     \setlength{\listparindent}{0mm}
     \setlength{\itemindent}{\listparindent}
     \setlength{\parsep}{0pt}
     \item\relax}}{\end{list}\vspace{.5\baselineskip}}
\begin{summary}
{\footnotesize {\bf Abstract.} 
  This article deals with a simultaneous abstract evolution equation. 
This includes a parabolic-hyperbolic phase-field system as an example  
which consists of a parabolic equation for the relative temperature  
coupled with a semilinear damped wave equation for the order parameter  
(see e.g., \cite{GP2003, GP2004, GPS2006, WGZ2007dynamicalBD, WGZ2007}). 
Although a time discretization of 
an initial value problem for
an abstract evolution equation has been studied (see e.g., \cite{CF1996}), 
time discretizations of initial value problems 
for simultaneous abstract evolution equations 
seem to be not studied yet. 
In this paper we focus on a time discretization of a simultaneous abstract evolution 
equation applying to parabolic-hyperbolic phase-field systems. 
Moreover, we can establish an error estimate for the difference between 
continuous and discrete solutions. 
}
\end{summary}
\vspace{10pt}

\newpage

\section{Introduction} \label{Sec1}

A time discretization of 
an initial value problem for
an abstract evolution equation 
has been studied. 
For example, Colli--Favini \cite{CF1996} have proved existence of solutions to 
the nonlinear Cauchy problem 
\begin{equation*}
     \begin{cases}
      L\dfrac{d^2u}{dt^2} + B\dfrac{du}{dt} + Au = g    
         & \mbox{in}\ (0, T), 
 \\[3mm] 
      u(0) = u_{0},\ \dfrac{du}{dt}(0) = w_{0}                                            
     \end{cases}
\end{equation*}
by employing a time discretization scheme, 
where $T>0$, $L : H \to H$ and $A : V \to V^{*}$ are linear positive selfadjoint operators, 
$H$ and $V$ are real Hilbert spaces, $V \subset H$, $V^{*}$ is the dual space of $V$, 
$B : V \to V^{*}$ is a maximal monotone operator, 
$g : (0, T) \to V^{*}$ and $u_{0}, w_{0} \in V$ are given.  
Moreover, they have derived an error estimate for the difference between 
continuous and discrete solutions. 
On the other hand, time discretizations of 
initial value problems for
simultaneous abstract evolution equations 
seem to be not studied yet. 

The system 
\begin{equation*}\tag*{(E)}\label{E}
     \begin{cases}
         (\theta + \lambda(\varphi))_{t} - \Delta\theta = f     
         & \mbox{in}\ \Omega\times(0, \infty), 
 \\[0mm]
        \ep\varphi_{tt} + \varphi_{t} - \Delta\varphi 
        + \eta(\varphi) = \lambda'(\varphi)\theta 
         & \mbox{in}\ \Omega\times(0, \infty),    
 \\[0.5mm]
     \theta(0) = \theta_{0},\ \varphi(0) = \varphi_{0},\ \varphi_{t}(0) = v_{0} 
     & \mbox{in}\ \Omega     
     \end{cases}
\end{equation*}
is a parabolic-hyperbolic phase-field system 
(see e.g., \cite{GP2003, GP2004, GPS2006, WGZ2007dynamicalBD, WGZ2007}),  
where $\Omega \subset \mathbb{R}^3$ is a bounded domain with 
smooth boundary,  
$\lambda$ and $\eta$ are smooth functions, 
$\ep>0$, $f$ is a time dependent heat source, 
and $\theta_{0}$, $\varphi_{0}$, $v_{0}$ are given initial data defined in $\Omega$.  
The unknown function $\theta$ is the relative temperature.   
The unknown function $\varphi$ is the order parameter.  
The function $\lambda$ has a quadratic growth, 
e.g, $\lambda(r) = ar^2 + br + c$ ($a, b, c \in \mathbb{R}$); 
while the function $\eta$ has a cubic growth, e.g., 
$\eta(r) = d_{1}r^3 - d_{2}r$ ($d_{1}, d_{2} > 0$). 
The second time derivative $\ep\varphi_{tt}$ is the inertial term 
which characterizes the hyperbolic dynamics. 
In the case that $\ep=0$  
the system \ref{E} is the classical phase-field model proposed by Caginalp 
(cf.\ \cite{Cag, EllZheng}; one may also see 
the monographs \cite{BrokSpr, fremond, V1996}). 
The system \ref{E} 
endowed with homogeneous Dirichlet--Neumann boundary conditions 
has been analyzed by e.g., 
Grasselli--Pata \cite{GP2003, GP2004}
and
Grasselli--Petzeltov\'a--Schimperna \cite{GPS2006}. 
Wu--Grasselli--Zheng \cite{WGZ2007} has studied the system \ref{E} 
with homogeneous Neumann boundary conditions for both $\theta$ and $\varphi$. 
In the case that $\lambda(r) = r$ for all $r \in \mathbb{R}$, 
the system \ref{E} with dynamical boundary condition 
has been analyzed by e.g., Wu--Grasselli--Zheng \cite{WGZ2007dynamicalBD}. 
In the case that $\ep=0$ and $\lambda(r) = r$ for all $r \in \mathbb{R}$  
Colli--K.\ \cite{CK} have employed a time discretization scheme 
to prove existence of solutions to the system \ref{E} 
under homogeneous Neumann--Neumann boundary conditions 
and established an error estimate for the difference between 
continuous and discrete solutions. 
However, time discretizations of parabolic-hyperbolic phase-field systems 
seem to be not studied yet. 

In this paper we consider 
the initial value problem for
the simultaneous abstract evolution equation 
%
%
 \begin{equation*}\tag*{(P)}\label{P}
     \begin{cases}
         \dfrac{d\theta}{dt} + \dfrac{d\varphi}{dt} + A_{1}\theta = f     
         & \mbox{in}\ (0, T), 
 \\[4mm]
 L\dfrac{d^2\varphi}{dt^2} + B\dfrac{d\varphi}{dt} + A_{2}\varphi
         + \Phi\varphi + {\cal L}\varphi = \theta   
         & \mbox{in}\ (0, T), 
 \\[2mm] 
       \theta(0) = \theta_{0},\ \varphi(0) = \varphi_{0},\ \dfrac{d\varphi}{dt}(0) = v_{0},    
     \end{cases}
 \end{equation*}
where $T>0$,  
$L : H \to H$ is a linear positive selfadjoint operator, 
$B : D(B) \subset H \to H$, $A_{j} : D(A_{j}) \subset H \to H$ ($j = 1, 2$) are 
linear maximal monotone selfadjoint  operators, 
$V_{j}$ ($j=1, 2$) are linear subspaces of $V$ 
satisfying $D(A_{j}) \subset V_{j}$ ($j=1, 2$), 
$\Phi : D(\Phi) \subset H \to H$ is a maximal monotone operator,   
${\cal L} : H \to H$ is a Lipschitz operator, 
$f : (0, T) \to H$ and $\theta_{0} \in V_{1}$, $\varphi_{0}, v_{0} \in V_{2}$ 
are given.  
Moreover, in reference to \cite{CF1996, CK}, we deal with the problem 
%
%
%
 \begin{equation*}\tag*{(P)$_{n}$}\label{Pn}
     \begin{cases}
      \delta_{h}\theta_{n} + \delta_{h}\varphi_{n} + A_{1}\theta_{n+1}= f_{n+1},  
         \\[1mm] 
       L z_{n+1} + Bv_{n+1} + A_{2}\varphi_{n+1} 
      + \Phi\varphi_{n+1} + {\cal L}\varphi_{n+1} = \theta_{n+1}, 
          \\[1mm] 
       z_{0}=z_{1},\ z_{n+1} = \delta_{h}v_{n}, 
         \\[1mm]
      v_{n+1}=\delta_{h}\varphi_{n}  
     \end{cases}
 \end{equation*}
for $n=0, ..., N-1$,  
where $h=\frac{T}{N}$, $N \in \mathbb{N}$, 
\begin{align}\label{deltah}
\delta_{h}\theta_{n} := \dfrac{\theta_{n+1}-\theta_{n}}{h}, \ 
\delta_{h}\varphi_{n} := \dfrac{\varphi_{n+1}-\varphi_{n}}{h}, \ 
\delta_{h}v_{n} := \dfrac{v_{n+1}-v_n}{h},  
\end{align}
and $f_{k} := \frac{1}{h}\int_{(k-1)h}^{kh}f(s)\,ds$ for $k=1, ..., N$.
Here, putting 
\begin{align}
& \widehat{\theta}_{h}(0) := \theta_{0},\  
   \frac{d\widehat{\theta}_{h}}{dt}(t) := \delta_{h}\theta_{n}, \quad
   \widehat{\varphi}_{h}(0) := \varphi_{0},\ 
   \frac{d\widehat{\varphi}_{h}}{dt}(t) := \delta_{h}\varphi_{n},    \label{hat1} 
\\[3mm] 
&\widehat{v}_{h}(0) := v_{0},\   
   \frac{d\widehat{v}_{h}}{dt}(t) := \delta_{h}v_{n}, \label{hat2} 
\\[2mm] 
&\overline{\theta}_{h}(t) := \theta_{n+1},\ \overline{z}_{h}(t) := z_{n+1},\ 
\overline{\varphi}_{h}(t) := \varphi_{n+1},\ \overline{v}_{h}(t) := v_{n+1},\ 
\overline{f}_{h}(t) := f_{n+1}      
\label{overandunderline}   
\end{align}
for a.a.\ $t \in (nh, (n+1)h)$, $n=0, ..., N-1$, 
we can rewrite \ref{Pn} as  
%
%
%
 \begin{equation*}\tag*{(P)$_{h}$}\label{Ph}
     \begin{cases}
       \dfrac{d\widehat{\theta}_{h}}{dt} + \dfrac{d\widehat{\varphi}_{h}}{dt}  
       + A_{1}\overline{\theta}_{h}= \overline{f}_{h} 
      \quad \mbox{in}\ (0, T),  
      \\[4.5mm]
         L\overline{z}_{h} + B\overline{v}_{h} + A_{2}\overline{\varphi}_{h} 
         + \Phi\overline{\varphi}_{h} + {\cal L}\overline{\varphi}_{h} 
         = \overline{\theta}_{h} 
     \quad \mbox{in}\ (0, T),  
     \\[3.5mm] 
        \overline{z}_{h} = \dfrac{d\widehat{v}_{h}}{dt},\ 
        \overline{v}_{h} = \dfrac{d\widehat{\varphi}_{h}}{dt} 
     \quad \mbox{in}\ (0, T),   
     \\[4.5mm] 
         z_{0} = \dfrac{v_{1}-v_{0}}{h},\  \widehat{\theta}_{h}(0)=\theta_{0},\ 
         \widehat{\varphi}_{h}(0) = \varphi_{0},\ \widehat{v}_{h}(0)=v_{0}. 
        
     \end{cases}
 \end{equation*}

\smallskip

\begin{remark}
Owing to  \eqref{hat1}-\eqref{overandunderline}, 
the reader can check directly the following identities: 
\begin{align}
&\|\widehat{\varphi}_h\|_{L^{\infty}(0, T; V_{2})} 
= \max\{\|\varphi_{0}\|_{V_{2}}, \|\overline{\varphi}_h\|_{L^{\infty}(0, T; V_{2})}\}, 
\label{rem1}
\\[6mm]
&\|\widehat{v}_h\|_{L^{\infty}(0, T; V_{2})} 
= \max\{\|v_{0}\|_{V_{2}}, \|\overline{v}_h\|_{L^{\infty}(0, T; V_{2})}\}, \label{rem2} 
\\[6mm] 
&\|\widehat{\theta}_h\|_{L^{\infty}(0, T; V_{1})} 
= \max\{\|\theta_{0}\|_{V_{1}}, 
                           \|\overline{\theta}_h\|_{L^{\infty}(0, T; V_{1})}\}, \label{rem3} 
\\[5mm] 
&\|\overline{\varphi}_h - \widehat{\varphi}_h\|_{L^{\infty}(0, T; V_{2})} 
= h\Bigl\|\frac{d\widehat{\varphi}_h}{dt}\Bigr\|_{L^{\infty}(0, T; V_{2})} 
= h\|\overline{v}_h\|_{L^{\infty}(0, T; V_{2})} , \label{rem4}  
\\[3mm] 
&\|\overline{v}_h - \widehat{v}_h\|_{L^{\infty}(0, T; H)} 
= h\Bigl\|\frac{d\widehat{v}_h}{dt}\Bigr\|_{L^{\infty}(0, T; H)} 
= h\|\overline{z}_h\|_{L^{\infty}(0, T; H)} , \label{rem5} 
\\[3mm] 
&\|\overline{\theta}_h - \widehat{\theta}_h\|_{L^2(0, T; H)}^2 
= \frac{h^2}{3}\Bigl\|\frac{d\widehat{\theta}_h}{dt}\Bigr\|_{L^2(0, T; H)}^2. \label{rem6}
\end{align}  
\end{remark} 
\smallskip

Moreover, we deal with the following conditions (C1)-(C14): 
%
%
%
 \begin{enumerate} 
 \setlength{\itemsep}{0mm}
 \item[(C1)] 
 $V$ and $H$ are real Hilbert spaces satisfying $V \subset H$ with 
 dense, continuous and compact embedding. Moreover, 
 the inclusions $V \subset H \subset V^{*}$ hold 
 by identifying $H$ with its dual space $H^{*}$, where 
 $V^{*}$ is the dual space of $V$.  
 \item[(C2)] 
 $V_{j}$ ($j=1, 2$) are closed linear subspaces of $V$, dense in $H$ 
 and reflexive.   
 \item[(C3)] 
 $L : H \to H$ is a bounded linear operator fulfilling 
    \[
    (Lw, z)_{H} = (w, Lz)_{H}\ \mbox{for all}\ w, z \in H, \quad  
    (Lw, w)_{H} \geq c_{L}\|w\|_{H}^2\ \mbox{for all}\ w \in H,   
    \]
 where $c_{L} > 0$ is a constant. 
 \item[(C4)]
 $A_{1} : D(A_{1}) \subset H \to H$ 
 is a linear maximal monotone selfadjoint operator, 
 where $D(A_{1})$ is a linear subspace of $H$ and $D(A_{1}) \subset V_{1}$.  
 Moreover, there exists a bounded linear monotone operator
 $A_{1}^{*} : V_{1} \to V_{1}^{*}$ such that  
 \begin{align*}
 &\langle A_{1}^{*}w, z \rangle_{V_{1}^{*}, V_{1}} 
 = \langle A_{1}^{*}z, w \rangle_{V_{1}^{*}, V_{1}} 
 \quad\mbox{for all}\ w, z \in V_{1}, 
 \\  
 &A_{1}^{*}w = A_{1}w \quad\mbox{for all}\ w \in D(A_{1}). 
 \end{align*}   
 Moreover, for all $\alpha>0$ there exists $\sigma_{\alpha}>0$ such that 
 $$
 \langle A_{1}^{*}w, w \rangle_{V_{1}^{*}, V_{1}} + \alpha\|w\|_{H}^2 
 \geq \sigma_{\alpha}\|w\|_{V_{1}}^2 
 \quad \mbox{for all}\ w \in V_{1}. 
 $$ 
 \item[(C5)]  
 For all $g \in H$ and all $a> 0$, if there exists 
 $\theta \in V_{1}$ such that 
 $\theta + aA_{1}^{*}\theta = g$ in $V_{1}^{*}$, 
 then it follows that $\theta \in D(A_{1})$ and 
 $\theta + aA_{1}\theta = g$ in $H$. 
 \item[(C6)] 
 $B : D(B) \subset H \to H$, $A_{2} : D(A_{2}) \subset H \to H$ are 
 linear maximal monotone selfadjoint  operators, 
 where $D(B)$ and $D(A_{2})$ are linear subspaces of $H$, 
 satisfying 
 \begin{align*}
 &D(B) \cap D(A_{2}) \neq \emptyset, 
 \\ 
 &(Bw, A_{2}w)_{H} \geq 0 \quad \mbox{for all}\ w \in D(B) \cap D(A_{2}), 
 \\
 &(Bw, A_{2}z)_{H} = (Bz, A_{2}w)_{H} 
                                    \quad \mbox{for all}\ w, z \in D(B) \cap D(A_{2}). 
 \end{align*} 
 Moreover, the inclusion $D(A_{2}) \subset V_{2}$ holds. 
 \item[(C7)] 
 $\Phi : D(\Phi) \subset H \to H$ is a maximal monotone operator 
 satisfying $\Phi(0) = 0$ and $V \subset D(\Phi)$. 
 Moreover, there exist constants $p, q, C_{\Phi} > 0$ such that  
 \[
 \|\Phi w - \Phi z\|_{H} \leq C_{\Phi}(1 + \|w\|_{V}^p + \|z\|_{V}^{q})\|w-z\|_{V}
 \quad \mbox{for all}\ w, z \in V. 
 \]
 \item[(C8)] 
 There exists a function $i : V \to \{x \in \mathbb{R}\ |\ x\geq0 \}$ 
 such that $(\Phi w, w-z)_{H} \geq i(w) - i(z)$ for all $w, z \in V$.
 \item[(C9)] 
 $\Phi_{\lambda}(0) = 0$, 
 $(\Phi_{\lambda}w, Bw)_{H} \geq 0$  
 for all $w \in D(B)$, 
 $(\Phi_{\lambda}w, A_{2}w)_{H} \geq 0$  
 for all $w \in D(A_{2})$, 
 where $\lambda > 0$ and $\Phi_{\lambda} : H \to H$ is  
 the Yosida approximation of $\Phi$. 
 \item[(C10)] 
 $B^{*} : V_{2} \to V_{2}^{*}$ is a bounded linear monotone operator fulfilling 
 \begin{align*}
 &\langle B^{*}w, z \rangle_{V_{2}^{*}, V_{2}} = \langle B^{*}z, w \rangle_{V_{2}^{*}, V_{2}}  
 \quad\mbox{for all}\ w, z \in V_{2},  
 \\  
 &B^{*}w = Bw \quad\mbox{for all}\ w \in D(B) \cap V_{2}. 
 \end{align*}
 \item[(C11)] 
 $A_{2}^{*} : V_{2} \to V_{2}^{*}$ is a bounded linear monotone operator fulfilling 
 \begin{align*}
 &\langle A_{2}^{*}w, z \rangle_{V_{2}^{*}, V_{2}} 
 = \langle A_{2}^{*}z, w \rangle_{V_{2}^{*}, V_{2}} 
 \quad\mbox{for all}\ w, z \in V_{2}, 
 \\  
 &A_{2}^{*}w = A_{2}w \quad\mbox{for all}\ w \in D(A_{2}). 
 \end{align*}   
 Moreover, for all $\alpha>0$ there exists $\omega_{\alpha}>0$ such that 
 $$
 \langle A_{2}^{*}w, w \rangle_{V_{2}^{*}, V_{2}} + \alpha\|w\|_{H}^2 
 \geq \omega_{\alpha}\|w\|_{V_{2}}^2 
 \quad \mbox{for all}\ w \in V_{2}. 
 $$ 
 \item[(C12)] 
 For all $g \in H$, $a, b, c, d > 0$, $\lambda>0$, if there exists 
 $\varphi_{\lambda} \in V_{2}$ such that 
 $$
 L\varphi_{\lambda} + aB^{*}\varphi_{\lambda} + bA_{2}^{*}\varphi_{\lambda} 
 + c\Phi_{\lambda}\varphi_{\lambda} + d{\cal L}\varphi_{\lambda} = g \quad  
  \mbox{in}\ V_{2}^{*},  
 $$ 
 then it follows that $\varphi_{\lambda} \in D(B) \cap D(A_{2})$ and 
 $$
 L\varphi_{\lambda} + aB\varphi_{\lambda} + bA_{2}\varphi_{\lambda} 
 + c\Phi_{\lambda}\varphi_{\lambda} + d{\cal L}\varphi_{\lambda} = g \quad 
  \mbox{in}\ H.   
 $$ 
 \item[(C13)] 
 ${\cal L} : H \to H$ is a Lipschitz continuous operator 
 with Lipschitz constant $C_{{\cal L}}>0$.
 \item[(C14)] 
 $\theta_{0} \in V_{1}$, $\varphi_{0} \in  D(B) \cap D(A_{2})$, 
$v_{0} \in D(B) \cap V_{2}$, 
 $f \in L^2(0, T; H)$.    
 \end{enumerate}
\begin{remark}
We set the conditions (C3), (C4) and (C11) 
in reference to \cite[Section 2]{CF1996}.  
The conditions (C5) and (C12) are equivalent to the elliptic regularity theory 
under some cases (see Section \ref{Sec2}). 
Moreover, we set the conditions (C7)-(C9) and (C13)  
by trying to keep a typical example (see Section \ref{Sec2})  
in reference to assumptions in 
\cite{CK, GP2003, GP2004, GPS2006, WGZ2007, WGZ2007dynamicalBD}.    
 
\end{remark}

\smallskip

%
%
%
We define solutions of \ref{P} as follows. 
\begin{df}
A pair $(\theta, \varphi)$ with 
\begin{align*}
&\theta \in H^1(0, T; H) \cap L^{\infty}(0, T; V_{1}) \cap L^2(0, T; D(A_{1})), 
\\ 
&\varphi \in W^{2, \infty}(0, T; H) \cap W^{1, \infty}(0, T; V_{2}) 
                                                                     \cap L^2(0, T; D(A_{2})), 
\\[1mm] 
&\frac{d\varphi}{dt} \in L^2(0, T; D(B)),\ \Phi\varphi \in L^{\infty}(0, T; H) 
\end{align*}
is called a solution of \ref{P} if $(\theta, \varphi)$ satisfies 
\begin{align}
&\dfrac{d\theta}{dt} + \dfrac{d\varphi}{dt} + A_{1}\theta = f  
    \quad\mbox{in}\ H   \quad \mbox{a.e.\ on}\ (0, T), \label{df1} 
\\[2mm] 
&L\dfrac{d^2\varphi}{dt^2} + B\dfrac{d\varphi}{dt} + A_{2}\varphi
         + \Phi\varphi + {\cal L}\varphi = \theta   
         \quad\mbox{in}\ H   \quad \mbox{a.e.\ on}\ (0, T), \label{df2} 
\\ 
&\theta(0) = \theta_{0},\ \varphi(0) = \varphi_{0},\ \dfrac{d\varphi}{dt}(0) = v_{0}                                
        \quad\mbox{in}\ H. \label{df3}
\end{align}
\end{df}

\medskip

Now the main results read as follows. 

\begin{thm}\label{maintheorem1}
Assume that {\rm (C1)-(C14)} hold. 
Then there exists $h_{0} \in (0, 1)$ such that  
for all $h \in (0, h_{0})$ 
there exists a unique solution $(\theta_{n+1}, \varphi_{n+1})$ 
of {\rm \ref{Pn}} 
satisfying 
$$
\theta_{n+1} \in D(A_{1}),\ \varphi_{n+1} \in D(B) \cap D(A_{2}) 
\quad\mbox{for}\ n=0, ..., N-1.
$$ 
\end{thm}

\begin{thm}\label{maintheorem2}
Let $h_{0}$ be as in Theorem \ref{maintheorem1}. 
Assume that {\rm (C1)-(C14)} hold. 
Then there exists a unique solution $(\theta, \varphi)$  
of {\rm \ref{P}}. 
\end{thm}

\begin{thm}\label{erroresti} 
Let $h_{0}$ be as in Theorem \ref{maintheorem1}. 
Assume that {\rm (C1)-(C14)} hold. 
Assume further that $f \in W^{1, 1}(0, T; H)$. 
Then there exist constants $h_{00} \in (0, h_{0})$ and 
$M=M(T)>0$ such that 
\begin{align*}
&\|L^{1/2}(\widehat{v}_{h} - v)\|_{L^{\infty}(0, T; H)} 
+ \|B^{1/2}(\overline{v}_{h}-v)\|_{L^2(0, T; H)} 
+ \|\widehat{\varphi}_{h} - \varphi\|_{L^{\infty}(0, T; V_{2})} 
\\ 
&+ \|\widehat{\theta}_{h} - \theta\|_{L^{\infty}(0, T; H)} 
+ \|\overline{\theta}_{h} - \theta\|_{L^2(0, T; V_{1})} 
\leq M h^{1/2}
\end{align*}
for all $h \in (0, h_{00})$, 
where $v = \frac{d\varphi}{dt}$. 
\end{thm}

This paper is organized as follows. 
Section \ref{Sec2} gives some examples. 
In Section \ref{Sec3} we establish existence of solutions to \ref{Pn} 
in reference to \cite[Section 4]{CK2}. 
Section \ref{Sec4} devotes to the proof of existence for \ref{P}. 
In Section \ref{Sec5} we derive error estimates between solutions of \ref{P} 
and solutions of \ref{Ph}. 

\vspace{10pt}

\section{Examples}\label{Sec2}
In this section we give the following examples. 
%
%
%
\begin{example}
We consider the following homogeneous Dirichlet--Neumann problem 
\begin{equation*}\tag*{(P1)}\label{P1}
     \begin{cases}
         \theta_{t} + \varphi_{t} - \Delta\theta = f     
         & \mbox{in}\ \Omega\times(0, T), 
 \\[1mm]
        \varphi_{tt} + \varphi_{t} - \Delta\varphi 
        + \beta(\varphi) + \pi(\varphi) = \theta 
         & \mbox{in}\ \Omega\times(0, T), 
 \\[1mm] 
         \theta = \partial_{\nu}\varphi = 0 
         & \mbox{on}\ \partial\Omega\times(0, T),  
 \\[1mm]
         \theta(0) = \theta_{0},\ \varphi(0) = \varphi_{0},\ \varphi_{t}(0) = v_{0}  
         
         &\mbox{in}\ \Omega,                                     
     \end{cases}
\end{equation*}
where $\Omega \subset \mathbb{R}^3$ is a bounded domain 
with smooth boundary $\partial\Omega$, $T>0$, 
under the following conditions: 
\begin{enumerate} 
\setlength{\itemsep}{2mm}
\item[(J1)] 
$\beta : \mathbb{R} \to \mathbb{R}$                                
is a single-valued maximal monotone function and 
there exists a proper differentiable (lower semicontinuous) convex function 
$\widehat{\beta} : \mathbb{R} \to [0, +\infty)$ such that 
$\widehat{\beta}(0) = 0$ and 
$\beta(r) = \widehat{\beta}\,'(r) = \partial\widehat{\beta}(r)$ 
for all $r \in \mathbb{R}$, 
where $\widehat{\beta}\,'$ and $\partial\widehat{\beta}$, respectively, 
are the differential and subdifferential of $\widehat{\beta}$.
\item[(J2)] 
$\beta \in C^2(\mathbb{R})$. 
Moreover, there exists a constant $C_{\beta} > 0$ such that 
$|\beta''(r)| \leq C_{\beta}(1+|r|)$ for all $r \in \mathbb{R}$. 
\item[(J3)] 
$\pi : \mathbb{R} \to \mathbb{R}$ is a Lipschitz continuous function. 
\item[(J4)] 
 $\theta_{0} \in H_{0}^{1}(\Omega)$, 
 $\varphi_{0} \in H^2(\Omega)$, 
 $\partial_{\nu}\varphi_{0} = 0$ a.e.\ on $\partial\Omega$, $v_{0} \in H^1(\Omega)$, 
 $f \in L^2(0, T; L^2(\Omega))$. 
\end{enumerate}
Moreover, we put 
\begin{align*}
&V:=H^1(\Omega),\ H:=L^2(\Omega),\ V_{1}:=H_{0}^{1}(\Omega),\ V_{2}:=H^1(\Omega), 
\\ 
&L := I : H \to H, 
\\ 
&A_{1} := - \Delta : D(A_{1}):=H^2(\Omega) \cap H_{0}^{1}(\Omega) \subset H \to H,  
\\ 
&B := I : D(B):=H \to H, 
\\ 
&A_{2} := - \Delta : D(A_{2}):=\{z \in H^2(\Omega)\ 
                           |\ \partial_{\nu}z = 0\quad \mbox{a.e.\ on}\ \partial\Omega\} 
                          \subset H  
                          \to H  
\end{align*}
and define the operators 
$A_{1}^{*} : V_{1} \to V_{1}^{*}$, $B^{*} : V_{2} \to V_{2}^{*}$, 
$A_{2}^{*} : V_{2} \to V_{2}^{*}$, $\Phi : D(\Phi) \subset H \to H$, 
${\cal L}: H \to H$
as 
\begin{align*}
&\langle A_{1}^{*}w, z \rangle_{V_{1}^{*}, V_{1}} 
:= \int_{\Omega} \nabla w \cdot \nabla z 
\quad \mbox{for}\ w, z \in V_{1}, 
\\[1mm] 
&\langle B^{*}w, z \rangle_{V_{2}^{*}, V_{2}} 
:= (w, z)_{H}
\quad \mbox{for}\ w, z \in V_{2}, 
\\[1mm] 
&\langle A_{2}^{*}w, z \rangle_{V_{2}^{*}, V_{2}} 
:= \int_{\Omega} \nabla w \cdot \nabla z 
\quad \mbox{for}\ w, z \in V_{2},  
\\[1mm]
&\Phi(z) := \beta(z) \quad \mbox{for}\ 
z \in D(\Phi) := \{z \in H\ |\ \beta(z) \in H \}, 
\\[1mm]
&{\cal L}(z) := \pi(z) \quad \mbox{for}\ z \in H. 
\end{align*}
Please note that the identity $\widehat{\beta}(0)=0$ in (J1) entails $\beta(0)=0$. 
We set (J1) in reference to an assumption in \cite{CK}. 
We assumed (J2) in reference to assumptions in 
\cite{GP2004, GPS2006, WGZ2007, WGZ2007dynamicalBD}.    
Moreover, we set (J3) in reference to assumptions in \cite{CK, GP2003}. 
Then the function $\mathbb{R} \ni r \mapsto d_{1}r^3-d_{2}r \in \mathbb{R}$ 
($d_{1}, d_{2} > 0$) is a typical example of $\beta + \pi$.   
Now we verify that $\Phi : D(\Phi) \subset H \to H$ is maximal monotone. 
We define the function $\phi : H \to \overline{\mathbb{R}}$ as 
\[
 \phi(z)=
   \begin{cases}
   \displaystyle\int_{\Omega}\hat{\beta}(z(x))\,dx 
   & \mbox{if}\ 
   z\in D(\phi)
              := \{z\in H\ |\ \hat{\beta}(z)\in L^1(\Omega) \}, 
   \\[2mm]
   +\infty 
   & \mbox{otherwise}.
   \end{cases}
\]
Then $\phi : H \to \overline{\mathbb{R}}$ is proper lower semicontinuous convex,  
whence $\partial\phi : D(\partial\phi) \subset H \to H$ 
is maximal monotone (see e.g., \cite[Theorem 2.8]{Barbu}). 
In addition, we have that 
\begin{align}
&D(\partial\phi) = \{z \in H\ |\ \beta(z) \in H \} = D(\Phi), 
\notag \\ \label{Phibeta}  
&\partial\phi(z)=\beta(z)=\Phi(z) \quad \mbox{for all}\ z \in D(\Phi)
\end{align}
(see e.g., \cite[Example 2.8.3]{Brezis}, \cite[Example II.8.B]{S-1997}). 
Thus $\Phi : D(\Phi) \subset H \to H$ is maximal monotone. 

Next we show that $\Phi_{\lambda}w = \beta_{\lambda}(w)$ for all $w \in H$, 
where $\beta_{\lambda}$ is the Yosida approximation operator of $\beta$ 
on $\mathbb{R}$. 
Since it follows from \eqref{Phibeta} that $\Phi=\partial\phi$, 
the identities  
\begin{align}\label{udon1}
\Phi_{\lambda}w = (\partial\phi)_{\lambda}w 
= \lambda^{-1}(w - J_{\lambda}^{\partial\phi}w) 
\end{align}
hold for all $w \in H$, 
where $J_{\lambda}^{\partial\phi} : H \to H$ 
is the resolvent operator of $\partial\phi$, 
that is, 
$J_{\lambda}^{\partial\phi}w = (I +\lambda\partial\phi)^{-1}w$ for all $w \in H$.  
On the other hand, since we derive from \eqref{Phibeta} that 
$\partial\phi(z) = \beta(z)$ for all $z \in D(\Phi)$, 
we can check that 
\begin{align}\label{udon2}
\lambda^{-1}(w - J_{\lambda}^{\partial\phi}w) 
=\lambda^{-1}(w - J_{\lambda}^{\beta}(w)) 
= \beta_{\lambda}(w) 
\end{align}
for all $w \in H$, where 
$J_{\lambda}^{\beta} : \mathbb{R} \to \mathbb{R}$ is 
the resolvent operator of $\beta$ on $\mathbb{R}$, 
that is, 
$J_{\lambda}^{\beta}(r) = (I +\lambda\beta)^{-1}(r)$ for all $r \in \mathbb{R}$.  
Hence combining \eqref{udon1} and \eqref{udon2} leads to the identity  
$\Phi_{\lambda}w = \beta_{\lambda}(w)$ for all $w \in H$. 

Next we prove that $V \subset D(\Phi)$ and 
 there exist constants $p, q, C_{\Phi} > 0$ such that  
 \[
 \|\Phi w - \Phi z\|_{H} \leq C_{\Phi}(1 + \|w\|_{V}^p + \|z\|_{V}^{q})\|w-z\|_{V}
 \]
for all $w, z \in V$. 
The Taylor theorem and the condition (J2) mean that  
\begin{align}
|\beta(r)-\beta(s)| 
&= \left|\beta'(s)(r-s) + \frac{1}{2}\beta''(r_{0})(r-s)^2\right|  
\notag \\ \label{au1}
&\leq |\beta'(s)||r-s| + \frac{C_{\beta}}{2}(1+|r|+|s|)(r-s)^2  
\end{align}
for all $r, s \in \mathbb{R}$, 
where $r_{0}$ is a constant belonging to $[r, s]$ or $[s, r]$.
Also, owing to the Taylor theorem and the condition (J2), it holds that  
\begin{align}
|\beta'(s)|=|\beta'(0) + \beta''(s_{0})s|
&\leq |\beta'(0)| + C_{\beta}(1+|s|)|s|
\notag \\ \label{au2}
&=  |\beta'(0)| + C_{\beta}(|s|+|s|^2)
\end{align}
for all $s \in \mathbb{R}$, where 
$s_{0} \in \mathbb{R}$ is a constant belonging to $[0, s]$ or $[s, 0]$.  
Thus we infer from \eqref{au1}, \eqref{au2} and 
the H\"older inequality 
that 
\begin{align}
&\|\beta(w)-\beta(z)\|_{H}^2 
\notag \\ 
&\leq C_{1}\|w-z\|_{L^2(\Omega)}^2 
      + C_{1}\|z(w-z)\|_{L^2(\Omega)}^2 
      +  C_{1}\|z^2(w-z)\|_{L^2(\Omega)}^2 
\notag \\ 
  &\,\quad+ C_{1}\|w-z\|_{L^4(\Omega)}^4 
      + C_{1}\|w(w-z)^2\|_{L^2(\Omega)}^2 
      + C_{1}\|z(w-z)^2\|_{L^2(\Omega)}^2 
\notag \\ 
&\leq C_{1}\|w-z\|_{L^2(\Omega)}^2 
         + C_{1}\|z\|_{L^4(\Omega)}^2\|w-z\|_{L^4(\Omega)}^2  
         + C_{1}\|z\|_{L^6(\Omega)}^4\|w-z\|_{L^6(\Omega)}^2 
\notag \\          
&\,\quad+ C_{1}\|w-z\|_{L^4(\Omega)}^4 
          + C_{1}\|w\|_{L^6(\Omega)}^2\|w-z\|_{L^6(\Omega)}^4 
          + C_{1}\|z\|_{L^6(\Omega)}^2\|w-z\|_{L^6(\Omega)}^4  \label{sky1}  
\end{align}
for all $w, z \in V$, where $C_{1}>0$ is a constant. 
Here the continuity of the embedding $V \hookrightarrow L^6(\Omega)$ 
and the boundedness of $\Omega$ imply that  
\begin{align}
&C_{1}\|w-z\|_{L^2(\Omega)}^2 
         + C_{1}\|z\|_{L^4(\Omega)}^2\|w-z\|_{L^4(\Omega)}^2  
         + C_{1}\|z\|_{L^6(\Omega)}^4\|w-z\|_{L^6(\Omega)}^2 
\notag \\          
&+ C_{1}\|w-z\|_{L^4(\Omega)}^4 
          + C_{1}\|w\|_{L^6(\Omega)}^2\|w-z\|_{L^6(\Omega)}^4 
          + C_{1}\|z\|_{L^6(\Omega)}^2\|w-z\|_{L^6(\Omega)}^4    
\notag \\ 
&\leq C_{2}\|w-z\|_{V}^2 
         + C_{2}\|z\|_{V}^2\|w-z\|_{V}^2  
         + C_{2}\|z\|_{V}^4\|w-z\|_{V}^2 
\notag \\          
&\,\quad+ C_{2}\|w-z\|_{V}^4 
          + C_{2}\|w\|_{V}^2\|w-z\|_{V}^4 
          + C_{2}\|z\|_{V}^2\|w-z\|_{V}^4   
\notag \\ 
&\leq C_{3}(1 + \|w\|_{V}^4 + \|z\|_{V}^4)\|w-z\|_{V}^2 \label{sky2}
\end{align}
for all $w, z \in V$, 
where $C_{2}=C_{2}(\Omega), C_{3}=C_{3}(\Omega)>0$ are some constants. 
Hence we deduce from \eqref{sky1} and \eqref{sky2} that 
\begin{align*} 
\|\beta(w)-\beta(z)\|_{H}^2 \leq C_{3}(1 + \|w\|_{V}^4 + \|z\|_{V}^4)\|w-z\|_{V}^2
\end{align*}
for all $w, z \in V$. 
Then, thanks to the identity $\beta(0)=0$, we have 
\begin{align*} 
\|\beta(w)\|_{H}^2 \leq C_{3}(1 + \|w\|_{V}^4)\|w\|_{V}^2
\end{align*}
for all $w \in V$. 
Therefore $V \subset D(\Phi)$ and 
 there exist constants $p, q, C_{\Phi} > 0$ such that  
 \[
 \|\Phi w - \Phi z\|_{H} \leq C_{\Phi}(1 + \|w\|_{V}^p + \|z\|_{V}^{q})\|w-z\|_{V}
 \]
for all $w, z \in V$. 

Next we confirm that 
there exists a function $i : V \to \{x \in \mathbb{R}\ |\ x\geq0 \}$ 
such that $(\Phi w, w-z)_{H} \geq i(w) - i(z)$ for all $w, z \in V$. 
We see from (J1) and the definition of the subdifferential that 
$\beta(r)(r-s) \geq \widehat{\beta}(r) - \widehat{\beta}(s)$ 
for all $r, s \in \mathbb{R}$. 
Thus, defining $i : V \to \{x \in \mathbb{R}\ |\ x\geq0 \}$ as 
\[
i(z) = \int_{\Omega} \widehat{\beta}(z)
\quad \mbox{for}\ z \in V \subset D(\Phi) \subset 
                  \{z \in H\ |\ \widehat{\beta}(z) \in L^1(\Omega) \}, 
\]
we can obtain that $(\Phi w, w-z)_{H} \geq i(w) - i(z)$ for all $w, z \in V$. 

Therefore the conditions (C1)-(C4), (C6)-(C11), (C13) and (C14) hold. 
Moreover, the elliptic regularity theory leads to (C5) and (C12). 
Similarly, we can check that 
the homogeneous Neumann--Neumann problem, 
the homogeneous  Dirichlet--Dirichlet problem  
and the homogeneous Neumann--Dirichlet problem 
are examples.  
\end{example}
\begin{example}
We can verify that the problem 
\begin{equation*}\tag*{(P2)}\label{P2}
     \begin{cases}
         \theta_{t} + \varphi_{t} - \Delta\theta = f     
         & \mbox{in}\ \Omega\times(0, T), 
 \\[1mm]
        \varphi_{tt} - \Delta\varphi_{t} - \Delta\varphi 
        + \beta(\varphi) + \pi(\varphi) = \theta 
         & \mbox{in}\ \Omega\times(0, T), 
 \\[1mm] 
         \theta = \varphi = 0 
         & \mbox{on}\ \partial\Omega\times(0, T),  
 \\[1mm]
         \theta(0) = \theta_{0},\ \varphi(0) = \varphi_{0},\ \varphi_{t}(0) = v_{0} 
         &\mbox{in}\ \Omega,                                     
     \end{cases}
\end{equation*}
where $\Omega \subset \mathbb{R}^3$ is a bounded domain 
with smooth boundary $\partial\Omega$, 
is an example under the three conditions (J1)-(J3) and 
the following condition 
\begin{enumerate} 
\item[(J5)] 
 $\theta_{0} \in H_{0}^{1}(\Omega)$, 
 $\varphi_{0} \in H^2(\Omega) \cap H_{0}^{1}(\Omega)$,
 $v_{0} \in H^2(\Omega) \cap H_{0}^{1}(\Omega)$, 
 $f \in L^2(0, T; L^2(\Omega))$. 
\end{enumerate}
Indeed, putting 
\begin{align*}
&V:=H^1(\Omega),\ H:=L^2(\Omega),\ V_{1}:=H_{0}^{1}(\Omega),\ 
V_{2}:=H_{0}^1(\Omega), 
\\ 
&L := I : H \to H, 
\\ 
&A_{1} := - \Delta : D(A_{1}):=H^2(\Omega) \cap H_{0}^{1}(\Omega)\subset H \to H, 
\\ 
&B := - \Delta : D(B):=H^2(\Omega) \cap H_{0}^{1}(\Omega)\subset H \to H, 
\\ 
&A_{2} := - \Delta : D(A_{2}):=H^2(\Omega) \cap H_{0}^{1}(\Omega)\subset H \to H 
\end{align*}
and defining the operators 
$A_{1}^{*} : V_{1} \to V_{1}^{*}$, $B^{*} : V_{2} \to V_{2}^{*}$, 
$A_{2}^{*} : V_{2} \to V_{2}^{*}$, $\Phi : D(\Phi) \subset H \to H$, 
${\cal L}: H \to H$
as 
\begin{align*}
&\langle A_{1}^{*}w, z \rangle_{V_{1}^{*}, V_{1}} 
:= \int_{\Omega} \nabla w \cdot \nabla z 
\quad \mbox{for}\ w, z \in V_{1}, 
\\[1mm] 
&\langle B^{*}w, z \rangle_{V_{2}^{*}, V_{2}} 
:= \int_{\Omega} \nabla w \cdot \nabla z 
\quad \mbox{for}\ w, z \in V_{2}, 
\\[1mm] 
&\langle A_{2}^{*}w, z \rangle_{V_{2}^{*}, V_{2}} 
:= \int_{\Omega} \nabla w \cdot \nabla z 
\quad \mbox{for}\ w, z \in V_{2},  
\\[1mm]
&\Phi(z) := \beta(z) \quad \mbox{for}\ 
z \in D(\Phi) := \{z \in H\ |\ \beta(z) \in H \}, 
\\[1mm]
&{\cal L}(z) := \pi(z) \quad \mbox{for}\ z \in H,  
\end{align*}
we can confirm that (C1)--(C14) hold. 
Similarly, we can show that 
the homogeneous Dirichlet--Neumann problem, 
the homogeneous  Neumann--Neumann problem  
and the homogeneous Neumann--Dirichlet problem 
are examples.  
\end{example}

\vspace{10pt}

\section{Existence of discrete solutions}\label{Sec3}

In this section we will prove Theorem \ref{maintheorem1}. 
\begin{lem}\label{elliptic1}
For all $g \in H$ and all $h >0$  
there exists a unique solution $\theta \in D(A_{1})$  
of the equation 
$\theta + hA_{1}\theta = g$ in $H$. 
\end{lem}
\begin{proof}
We define the operator $\Psi : V_{1} \to V_{1}^{*}$ as   
\begin{align*}
\langle \Psi\theta, w \rangle_{V_{1}^{*}, V_{1}} 
:= (\theta, w)_{H} + h\langle A_{1}^{*}\theta, w \rangle_{V_{1}^{*}, V_{1}} 
\quad \mbox{for}\ \theta, w \in V_{1}. 
\end{align*}
Then, owing to (C4), this operator $\Psi : V_{1} \to V_{1}^{*}$ is monotone, 
continuous and coercive, and then is  
surjective for all $h > 0$ (see e.g., \cite[p.\ 37]{Barbu}).  
Therefore the condition (C5) leads to Lemma \ref{elliptic1}.  
\end{proof}

\begin{lem}\label{elliptic2}
There exists 
$h_{1} \in \left(0, \Bigl(\frac{c_{L}}{1+C_{{\cal L}}} \Bigr)^{1/2}\right)$ such that 
for all $g \in H$ and all $h \in (0, h_{1})$ 
there exists a unique solution $\varphi \in D(B) \cap D(A_{2})$ 
of the equation 
\[
L\varphi + hB\varphi + h^2 A_{2}\varphi 
+ h^2 \Phi\varphi + h^2 {\cal L}\varphi = g 
\quad \mbox{in}\ H. 
\]
\end{lem}
\begin{proof}
We define the operator $\Psi : V_{2} \to V_{2}^{*}$ as   
\begin{align*}
\langle \Psi\varphi, w \rangle_{V_{2}^{*}, V_{2}} 
&:= (L\varphi, w)_{H} + h\langle B^{*}\varphi, w \rangle_{V_{2}^{*}, V_{2}} 
\notag \\ 
&\,\quad+ h^2\langle A_{2}^{*}\varphi, w \rangle_{V_{2}^{*}, V_{2}} 
+ h^2(\Phi_{\lambda}\varphi, w)_{H} 
    + h^2({\cal L}\varphi, w)_{H} 
\quad \mbox{for}\ \varphi, w \in V_{2}. 
\end{align*}
Then we see that this operator $\Psi : V_{2} \to V_{2}^{*}$ is monotone, continuous 
and coercive 
for all $h \in \left(0, \Bigl(\frac{c_{L}}{1+C_{{\cal L}}} \Bigr)^{1/2}\right)$. 
Indeed, it follows from (C3), (C11), the monotonicity of $B^{*}$ and $\Phi_{\lambda}$, 
and (C13) that 
\begin{align*}
&\langle \Psi\varphi - \Psi\overline{\varphi}, 
                                             \varphi-\overline{\varphi} \rangle_{V_{2}^{*}, V_{2}} 
\notag \\ 
&= (L(\varphi-\overline{\varphi}), \varphi-\overline{\varphi})_{H} 
    + h\langle B^{*}(\varphi-\overline{\varphi}), 
                                        \varphi-\overline{\varphi} \rangle_{V_{2}^{*}, V_{2}} 
    + h^2\langle A_{2}^{*}(\varphi-\overline{\varphi}), 
                                    \varphi-\overline{\varphi} \rangle_{V_{2}^{*}, V_{2}} 
\notag \\ 
    &\,\quad+ h^2(\Phi_{\lambda}\varphi - \Phi_{\lambda}\overline{\varphi}, 
                                                                      \varphi-\overline{\varphi})_{H} 
    + h^2({\cal L}\varphi - {\cal L}\overline{\varphi}, \varphi-\overline{\varphi})_{H} 
\notag \\ 
&\geq c_{L}\|\varphi-\overline{\varphi}\|_{H}^2 
         + \omega_{1}h^2\|\varphi-\overline{\varphi}\|_{V_{2}}^2 
         - h^2\|\varphi-\overline{\varphi}\|_{H}^2 
         - C_{{\cal L}}h^2\|\varphi-\overline{\varphi}\|_{H}^2 
\notag \\ 
&\geq \omega_{1}h^2\|\varphi-\overline{\varphi}\|_{V_{2}}^2 
\end{align*}
for all $\varphi, \overline{\varphi} \in V_{2}$ 
and all $h \in \left(0, \Bigl(\frac{c_{L}}{1+C_{{\cal L}}} \Bigr)^{1/2}\right)$. 
The boundedness of the operators $L : H \to H$, $B^{*} : V_{2} \to V_{2}^{*}$, 
$A_{2}^{*} : V_{2} \to V_{2}^{*}$,  
the Lipschitz continuity of $\Phi_{\lambda}$, the condition (C13) 
and the continuity of the embedding $V_{2} \hookrightarrow H$ 
yield that 
there exists a constant $C_{1} = C_{1}(\lambda) > 0$ such that  
\begin{align*}
&|\langle \Psi\varphi - \Psi\overline{\varphi}, w \rangle_{V_{2}^{*}, V_{2}}| 
\notag \\ 
&\leq |(L(\varphi-\overline{\varphi}), w)_{H}| 
    + h|\langle B^{*}(\varphi-\overline{\varphi}), w \rangle_{V_{2}^{*}, V_{2}}| 
    + h^2|\langle A_{2}^{*}(\varphi-\overline{\varphi}), w \rangle_{V_{2}^{*}, V_{2}}| 
\notag \\ 
    &\,\quad+ h^2|(\Phi_{\lambda}\varphi - \Phi_{\lambda}\overline{\varphi}, w)_{H}| 
    + h^2|({\cal L}\varphi - {\cal L}\overline{\varphi}, w)_{H}| 
\notag \\ 
&\leq C_{1}(1 + h + h^2)\|\varphi-\overline{\varphi}\|_{V_{2}}\|w\|_{V_{2}}
\end{align*}
for all $\varphi, \overline{\varphi}, w \in V_{2}$ and all $h>0$.  
Also, we have that  
$\langle \Psi\varphi - {\cal L}0, \varphi \rangle_{V_{2}^{*}, V_{2}} 
\geq \omega_{1}h^2\|\varphi\|_{V_{2}}^2$ 
for all $\varphi \in V_{2}$
and all $h \in \left(0, \Bigl(\frac{c_{L}}{1+C_{{\cal L}}} \Bigr)^{1/2}\right)$. 
Thus the operator $\Psi : V_{2} \to V_{2}^{*}$ is 
surjective for all $h \in \left(0, \Bigl(\frac{c_{L}}{1+C_{{\cal L}}} \Bigr)^{1/2}\right)$ 
(see e.g., \cite[p.\ 37]{Barbu}), whence we can deduce from (C12) that 
for all $g \in H$ and 
all $h \in \left(0, \Bigl(\frac{c_{L}}{1+C_{{\cal L}}} \Bigr)^{1/2}\right)$ 
there exists a unique solution $\varphi_{\lambda} \in D(B) \cap D(A)$ 
of the equation 
\begin{align}\label{BJ1}
L\varphi_{\lambda} + hB\varphi_{\lambda} + h^2A_{2}\varphi_{\lambda} 
+ h^2\Phi_{\lambda}\varphi_{\lambda} + h^2{\cal L}\varphi_{\lambda} 
= g   \quad \mbox{in}\ H. 
\end{align}
Here we multiply \eqref{BJ1} by $\varphi_{\lambda}$ 
and use the Young inequality, (C13) to infer that 
\begin{align*}
&(L\varphi_{\lambda}, \varphi_{\lambda})_{H} 
+ h(B\varphi_{\lambda}, \varphi_{\lambda})_{H} 
+ h^2 \langle A_{2}^{*}\varphi_{\lambda}, \varphi_{\lambda} \rangle_{V_{2}^{*}, V_{2}}   
+ h^2(\Phi_{\lambda}\varphi_{\lambda}, \varphi_{\lambda})_{H} 
\notag \\ 
&= (g, \varphi_{\lambda})_{H} 
     - h^2({\cal L}\varphi_{\lambda} - {\cal L}0, \varphi_{\lambda})_{H} 
     - h^2({\cal L}0, \varphi_{\lambda})_{H}  
\notag \\ 
&\leq \frac{c_{L}}{2}\|\varphi_{\lambda}\|_{H}^2 
         + \frac{1}{2c_{L}}\|g\|_{H}^2 
         + C_{{\cal L}}h^2\|\varphi_{\lambda}\|_{H}^2 
         + \frac{\|{\cal L}0\|_{H}^2}{2}h^2 
         + \frac{1}{2}h^2\|\varphi_{\lambda}\|_{H}^2.    
\end{align*}
Then, by (C3), (C11), the monotonicity of $B$ and $\Phi_{\lambda}$, 
there exists $h_{1} \in \left(0, \Bigl(\frac{c_{L}}{1+C_{{\cal L}}} \Bigr)^{1/2}\right)$ 
such that for all $h \in (0, h_{1})$ there exists a constant $C_{2} = C_{2}(h) > 0$ 
satisfying 
\begin{align}\label{BJ2}
\|\varphi_{\lambda}\|_{V_{2}}^2 \leq C_{2}   
\end{align}
for all $\lambda>0$. 
We derive from \eqref{BJ1}, (C9) and the Young inequality that 
\begin{align*}
h^2\|\Phi_{\lambda}\varphi_{\lambda}\|_{H}^2 
&= (g, \Phi_{\lambda}\varphi_{\lambda})_{H} 
     - (L\varphi_{\lambda}, \Phi_{\lambda}\varphi_{\lambda})_{H} 
    - h(B\varphi_{\lambda}, \Phi_{\lambda}\varphi_{\lambda})_{H} 
    - h^2(A_{2}\varphi_{\lambda}, \Phi_{\lambda}\varphi_{\lambda})_{H} 
\notag \\ 
&\,\quad- h^2({\cal L}\varphi_{\lambda}, \Phi_{\lambda}\varphi_{\lambda})_{H} 
\notag \\ 
&\leq \frac{3}{2h^2}\|g\|_{H}^2 + \frac{3}{2h^2}\|L\varphi_{\lambda}\|_{H}^2 
         + \frac{3}{2}h^2\|{\cal L}\varphi_{\lambda}\|_{H}^2 
         + \frac{1}{2}h^2\|\Phi_{\lambda}\varphi_{\lambda}\|_{H}^2.  
\end{align*}
Hence, thanks to the boundedness of the operator $L : H \to H$, (C13) and \eqref{BJ2}, 
we can verify that 
for all $h \in (0, h_{1})$ there exists a constant $C_{3} = C_{3}(h) > 0$ such that 
\begin{align}\label{BJ3}
\|\Phi_{\lambda}\varphi_{\lambda}\|_{H}^2 \leq C_{3}   
\end{align}
for all $\lambda>0$. 
We can confirm that 
\begin{align*}
h\|B\varphi_{\lambda}\|_{H}^2 
&= (g, B\varphi_{\lambda})_{H} 
     - (L\varphi_{\lambda}, B\varphi_{\lambda})_{H} 
    - h^2(A_{2}\varphi_{\lambda}, B\varphi_{\lambda})_{H} 
    - h^2(\Phi_{\lambda}\varphi_{\lambda}, B\varphi_{\lambda})_{H} 
\notag \\ 
&\,\quad- h^2({\cal L}\varphi_{\lambda}, B\varphi_{\lambda})_{H} 
\end{align*}
by \eqref{BJ1} 
and then the boundedness of the operator $L : H \to H$, 
(C6), (C9), (C13), the Young inequality and \eqref{BJ2} 
imply that for all $h \in (0, h_{1})$  
there exists a constant $C_{4} = C_{4}(h) > 0$ satisfying 
\begin{align}\label{BJ4}
\|B\varphi_{\lambda}\|_{H}^2 \leq C_{4}(h)   
\end{align}
for all $\lambda>0$. 
We see from \eqref{BJ1}-\eqref{BJ4} that 
for all $h \in (0, h_{1})$ 
there exists a constant $C_{5} = C_{5}(h) > 0$ such that 
\begin{align}\label{BJ5}
\|A_{2}\varphi_{\lambda}\|_{H}^2 \leq C_{5}(h)   
\end{align}
for all $\lambda>0$. 
Thus by \eqref{BJ2}-\eqref{BJ5} 
there exist $\varphi \in D(B) \cap D(A_{2})$ and $\xi \in H$ such that 
\begin{align}
\label{ellipweak1} 
&\varphi_{\lambda} \to \varphi \quad \mbox{weakly in}\ V_{2},  \\ 
\label{ellipweak1'} 
&L\varphi_{\lambda} \to L\varphi \quad \mbox{weakly in}\ H,  \\ 
\label{ellipweak2}
&\Phi_{\lambda}(\varphi_{\lambda}) \to \xi  
                                             \quad \mbox{weakly in}\ H,  \\ 
\label{ellipweak3}
&B\varphi_{\lambda} \to B\varphi \quad \mbox{weakly in}\ H,  \\ 
\label{ellipweak4}
&A_{2}\varphi_{\lambda} \to A_{2}\varphi 
                                              \quad \mbox{weakly in}\ H 
\end{align}
as $\lambda = \lambda_{j} \to +0$. 
Here the inequality \eqref{BJ2}, the convergence \eqref{ellipweak1} and 
the compactness of the embedding $V_{2} \hookrightarrow H$ 
yield that 
\begin{align}\label{ellipstrong} 
\varphi_{\lambda} \to \varphi \quad \mbox{strongly in}\ H
\end{align}
as $\lambda = \lambda_{j} \to +0$. 
Moreover, we have from \eqref{ellipweak2} and \eqref{ellipstrong} that  
$(\Phi_{\lambda}\varphi_{\lambda}, \varphi_{\lambda})_{H} 
\to (\xi, \varphi)_{H}$  
as $\lambda = \lambda_{j} \to +0$. 
Hence the inclusion and the identity 
\begin{align}\label{BJ6}
\varphi \in D(\Phi),\ \xi = \Phi\varphi 
\end{align}
hold (see e.g., \cite[Lemma 1.3, p.\ 42]{Barbu1}). 

Therefore, by virtue of \eqref{BJ1}, \eqref{ellipweak1'}-\eqref{BJ6} and (C13), 
we can check that there exists a solution $\varphi \in D(B) \cap D(A_{2})$ 
of the equation  
\[
L\varphi + hB\varphi + h^2 A_{2}\varphi 
+ h^2 \Phi\varphi + h^2 {\cal L}\varphi = g 
\quad \mbox{in}\ H. 
\] 
Moreover, owing to (C3), (C11), the monotonicity of $B$ and $\Phi$, and (C13), 
the solution $\varphi$ of this problem is unique. 
\end{proof}

\begin{prth1.1}
Let $h_{1}$ be as in Lemma \ref{elliptic2} and let $h \in (0, h_{1})$. 
Then we infer from \eqref{deltah}, the linearity of 
the operators $A_{1}$, $L$, $B$ and $A_{2}$ that 
the problem \ref{Pn} can be written as 
\begin{equation*}\tag*{(Q)$_{n}$}\label{Qn}
     \begin{cases}
      \theta_{n+1} + hA_{1}\theta_{n+1} 
      = \theta_{n} + \varphi_{n} + hf_{n+1} - \varphi_{n+1}, 
      \\[3mm] 
      L\varphi_{n+1} + hB\varphi_{n+1} + h^2A_{2}\varphi_{n+1} 
        + h^2\Phi\varphi_{n+1} + h^2{\cal L}\varphi_{n+1} 
      \\ 
      = L\varphi_{n} + hLv_{n} + hB\varphi_{n} + h^2\theta_{n+1},  
     \end{cases}
\end{equation*}
whence proving Theorem \ref{maintheorem1} is equivalent to 
establish existence and uniqueness of solutions to \ref{Qn} for $n=0, ..., N-1$. 
It suffices to consider the case that $n=0$. 
Thanks to Lemma \ref{elliptic1}, 
we can verify that 
for all $\varphi \in H$ 
there exists a unique solution $\overline{\theta} \in H$ of the equation
\begin{align}\label{popo1}
\overline{\theta} + hA_{1}\overline{\theta} 
= \theta_{0} + \varphi_{0} + hf_{1} - \varphi.   
\end{align}
Also, Lemma \ref{elliptic2} means that 
for all $\theta \in H$ there exists a unique solution $\overline{\varphi} \in H$ 
of the equation  
\begin{align}\label{popo2}
L\overline{\varphi} + hB\overline{\varphi} + h^2A_{2}\overline{\varphi}  
        + h^2\Phi\overline{\varphi} + h^2{\cal L}\overline{\varphi}  
= L\varphi_{0} + hLv_{0} + hB\varphi_{0} + h^2\theta.  
\end{align}
Therefore we can define the operators 
${\cal A} : H \to H$, ${\cal B} : H \to H$ 
and ${\cal S} : H \to H$ as 
$$
{\cal A}(\varphi) = \overline{\theta},\ 
{\cal B}(\theta) = \overline{\varphi} 
\quad \mbox{for}\ \varphi, \theta \in H 
$$
and 
$$
{\cal S} = {\cal B} \circ {\cal A},   
$$
respectively. 
Then we see from \eqref{popo1} and the Young inequality that   
\begin{align*}
&\|{\cal A}\varphi - {\cal A}\zeta\|_{H}^2  
+ h(A_{1}({\cal A}\varphi-{\cal A}\zeta), {\cal A}\varphi - {\cal A}\zeta)_{H} 
\notag \\ 
&= -(\varphi-\zeta, {\cal A}\varphi - {\cal A}\zeta)_{H} 
\leq   \frac{1}{2}\|\varphi-\zeta\|_{H}^2 
          + \frac{1}{2}\|{\cal A}\varphi - {\cal A}\zeta\|_{H}^2   
\end{align*}
for all $\varphi \in H$ and all $\zeta \in H$, 
and hence the inequality 
\begin{align}\label{popo3}
\|{\cal A}\varphi - {\cal A}\zeta\|_{H} 
\leq \|\varphi-\zeta\|_{H} 
\end{align}
holds for all $\varphi \in H$ and all $\zeta \in H$ 
by the monotonicity of $A_{1}$. 
On the other hand, since we derive from \eqref{popo2}, (C13) 
and the Young inequality that 
\begin{align*}
&(L({\cal S}\varphi-{\cal S}\zeta), {\cal S}\varphi-{\cal S}\zeta)_{H} 
+ h(B({\cal S}\varphi-{\cal S}\zeta), {\cal S}\varphi-{\cal S}\zeta)_{H} 
\notag \\ 
&+ h^2(A_{2}({\cal S}\varphi-{\cal S}\zeta), {\cal S}\varphi-{\cal S}\zeta)_{H} 
+ h^2(\Phi{\cal S}\varphi-\Phi{\cal S}\zeta, {\cal S}\varphi-{\cal S}\zeta)_{H} 
\notag \\ 
&= h^2({\cal A}\varphi-{\cal A}\zeta, {\cal S}\varphi-{\cal S}\zeta)_{H} 
     - h^2({\cal L}{\cal S}\varphi-{\cal L}{\cal S}\zeta, 
                                                           {\cal S}\varphi-{\cal S}\zeta)_{H} 
\notag \\ 
&\leq \frac{h^2}{4}\|{\cal A}\varphi-{\cal A}\zeta\|_{H}^2 
         + h^2\|{\cal S}\varphi-{\cal S}\zeta\|_{H}^2 
         + C_{{\cal L}}h^2\|{\cal S}\varphi-{\cal S}\zeta\|_{H}^2 
\end{align*}
for all $\varphi \in H$ and all $\zeta \in H$, 
it follows from (C3), the monotonicity of $B$, $A_{2}$ and $\Phi$ that 
\begin{align}\label{popo4}
\|{\cal S}\varphi - {\cal S}\zeta\|_{H} 
\leq \frac{h}{2(c_{L}  - h^2 - C_{{\cal L}}h^2)^{1/2}}
                                                  \|{\cal A}\varphi - {\cal A}\zeta\|_{H} 
\end{align}
for all $\varphi, \zeta \in H$ and all $h \in (0, h_{1})$. 
Hence, combining \eqref{popo3} and \eqref{popo4}, we have that  
\begin{align*}
\|{\cal S}\varphi - {\cal S}\zeta\|_{H} 
\leq \frac{h}{2(c_{L}  - h^2 - C_{{\cal L}}h^2)^{1/2}}\|\varphi - \zeta\|_{H} 
\end{align*}
for all $\varphi, \zeta \in H$ and all $h \in (0, h_{1})$. 
Therefore 
there exists $h_{0} \in (0, \min\{1, h_{1}\})$ such that 
the operator ${\cal S} : H \to H$ is a contraction  mapping 
for all $h \in (0, h_{0})$. 
Then the Banach fixed-point theorem yields that 
the operator ${\cal S} : H \to H$ has a unique fixed point,  
$\varphi_{1} = {\cal S}\varphi_{1} \in D(B) \cap D(A_{2})$. 
Thus, putting $\theta_{1} := {\cal A}\varphi_{1} \in D(A_{1})$, 
we can conclude that there exists a unique solution $(\theta_{1}, \varphi_{1})$ 
of \ref{Qn} in the case that $n=0$. \qed  
\end{prth1.1}

\vspace{10pt}
 
\section{Uniform estimates for \ref{Ph} and passage to the limit}\label{Sec4}
In this section we will establish a priori estimates for \ref{Ph}
and will prove Theorem \ref{maintheorem2} 
by passing to the limit in \ref{Ph} as $h\to+0$.

\begin{lem}\label{lemkuri1}
Let $h_{0}$ be as in Theorem \ref{maintheorem1}. 
Then there exist constants 
$h_{2} \in (0, h_{0})$ and $C=C(T)>0$ such that  
\begin{align*}
&\|\overline{v}_{h}\|_{L^{\infty}(0, T; H)}^2 
+ h\|\overline{z}_{h}\|_{L^2(0, T; H)}^2 
+ \|B^{1/2}\overline{v}_{h}\|_{L^2(0, T; H)}^2   
+ \|\overline{\varphi}_{h}\|_{L^{\infty}(0, T; V_{2})}^2 
\\
&+ h\|\overline{v}_{h}\|_{L^2(0, T; V_{2})}^2  
+ \|\overline{\theta}_{h}\|_{L^{\infty}(0, T; H)}^2 
+ h\Bigl\|\frac{d\widehat{\theta}_{h}}{dt}\Bigr\|_{L^2(0, T; H)}^2 
+ \|\overline{\theta}_{h}\|_{L^2(0, T; V_{1})}^2   
\leq C 
\end{align*}
for all $h \in (0, h_{2})$. 
\end{lem}
\begin{proof}
Multiplying the second equation in \ref{Pn} 
by $hv_{n+1} (= \varphi_{n+1}-\varphi_{n})$ 
and recalling \eqref{deltah} 
lead to the identity 
\begin{align}\label{a1}
&(L(v_{n+1}-v_{n}), v_{n+1})_{H} + h(Bv_{n+1}, v_{n+1})_{H} 
+ \langle A_{2}^{*}\varphi_{n+1}, \varphi_{n+1}-\varphi_{n} \rangle_{V_{2}^{*}, V_{2}}  
\notag \\ 
&+ (\varphi_{n+1}, \varphi_{n+1}-\varphi_{n})_{H} 
+ (\Phi\varphi_{n+1}, \varphi_{n+1}-\varphi_{n})_{H} 
\notag \\ 
&= h(\theta_{n+1}, v_{n+1})_{H} - h({\cal L}\varphi_{n+1}, v_{n+1})_{H} 
     + h(\varphi_{n+1}, v_{n+1})_{H}. 
\end{align}
Here we infer that 
\begin{align}\label{a2}
&(L(v_{n+1}-v_{n}), v_{n+1})_{H} 
\notag \\ 
&= (L^{1/2}(v_{n+1}-v_{n}), L^{1/2}v_{n+1})_{H} 
\notag \\ 
&=\frac{1}{2}\|L^{1/2}v_{n+1}\|_{H}^2 - \frac{1}{2}\|L^{1/2}v_{n}\|_{H}^2 
    + \frac{1}{2}\|L^{1/2}(v_{n+1}-v_{n})\|_{H}^2  
\end{align}
and 
\begin{align}\label{a3}
&\langle A_{2}^{*}\varphi_{n+1}, \varphi_{n+1}-\varphi_{n} \rangle_{V_{2}^{*}, V_{2}}  
+ (\varphi_{n+1}, \varphi_{n+1}-\varphi_{n})_{H} 
\notag \\ 
&= \frac{1}{2}\langle A_{2}^{*}\varphi_{n+1}, \varphi_{n+1} \rangle_{V_{2}^{*}, V_{2}} 
     - \frac{1}{2}\langle A_{2}^{*}\varphi_{n}, \varphi_{n} \rangle_{V_{2}^{*}, V_{2}} 
     + \frac{1}{2}\langle A_{2}^{*}(\varphi_{n+1}-\varphi_{n}), 
                                               \varphi_{n+1}-\varphi_{n} \rangle_{V_{2}^{*}, V_{2}}  
\notag \\ 
&\,\quad + \frac{1}{2}\|\varphi_{n+1}\|_{H}^2 - \frac{1}{2}\|\varphi_{n}\|_{H}^2 
    + \frac{1}{2}\|\varphi_{n+1}-\varphi_{n}\|_{H}^2. 
\end{align}
Hence we deduce 
from \eqref{a1}-\eqref{a3}, (C8), (C13), 
the continuity of the embedding $V_{2} \hookrightarrow H$ 
and the Young inequality that 
there exist constants $C_{1}, C_{2}>0$ such that 
\begin{align}\label{a4}
&\frac{1}{2}\|L^{1/2}v_{n+1}\|_{H}^2 - \frac{1}{2}\|L^{1/2}v_{n}\|_{H}^2 
    + \frac{1}{2}\|L^{1/2}(v_{n+1}-v_{n})\|_{H}^2 
    + h\|B^{1/2}v_{n+1}\|_{H}^2  
\notag \\
&+ \frac{1}{2}\langle A_{2}^{*}\varphi_{n+1}, \varphi_{n+1} \rangle_{V_{2}^{*}, V_{2}} 
     - \frac{1}{2}\langle A_{2}^{*}\varphi_{n}, \varphi_{n} \rangle_{V_{2}^{*}, V_{2}} 
     + \frac{1}{2}\langle A_{2}^{*}(\varphi_{n+1}-\varphi_{n}), 
                                               \varphi_{n+1}-\varphi_{n} \rangle_{V_{2}^{*}, V_{2}}  
\notag \\ 
&+ \frac{1}{2}\|\varphi_{n+1}\|_{H}^2 - \frac{1}{2}\|\varphi_{n}\|_{H}^2 
    + \frac{1}{2}\|\varphi_{n+1}-\varphi_{n}\|_{H}^2 
   + i(\varphi_{n+1}) - i(\varphi_{n}) 
\notag 
\\
&\leq \frac{1}{2}h\|\theta_{n+1}\|_{H}^2 + \frac{3}{2}h\|v_{n+1}\|_{H}^2 
        + C_{1}h\|\varphi_{n+1}\|_{V_{2}}^2 + C_{2}h
\end{align}
for all $h \in (0, h_{0})$. 
On the other hand, multiplying the first equation in \ref{Pn} by $h\theta_{n+1}$, 
we see from the Young inequality that 
\begin{align}\label{a5}
&\frac{1}{2}\|\theta_{n+1}\|_{H}^2 - \frac{1}{2}\|\theta_{n}\|_{H}^2 
    + \frac{1}{2}\|\theta_{n+1}-\theta_{n}\|_{H}^2 
+ h(A_{1}\theta_{n+1}, \theta_{n+1})_{H} 
\notag \\ 
&= h(f_{n+1}, \theta_{n+1})_{H} - h(v_{n+1}, \theta_{n+1})_{H} 
\notag \\ 
&\leq \frac{1}{2}h\|f_{n+1}\|_{H}^2 + \frac{1}{2}h\|v_{n+1}\|_{H}^2 
         + h\|\theta_{n+1}\|_{H}^2.  
\end{align}
Thus combining \eqref{a4} and \eqref{a5} implies that 
\begin{align}\label{a6}
&\frac{1}{2}\|L^{1/2}v_{n+1}\|_{H}^2 - \frac{1}{2}\|L^{1/2}v_{n}\|_{H}^2 
    + \frac{1}{2}\|L^{1/2}(v_{n+1}-v_{n})\|_{H}^2 + h\|B^{1/2}v_{n+1}\|_{H}^2  
\notag \\
&+ \frac{1}{2}\langle A_{2}^{*}\varphi_{n+1}, \varphi_{n+1} \rangle_{V_{2}^{*}, V_{2}} 
     - \frac{1}{2}\langle A_{2}^{*}\varphi_{n}, \varphi_{n} \rangle_{V_{2}^{*}, V_{2}} 
     + \frac{1}{2}\langle A_{2}^{*}(\varphi_{n+1}-\varphi_{n}), 
                                               \varphi_{n+1}-\varphi_{n} \rangle_{V_{2}^{*}, V_{2}}  
\notag \\ 
&+ \frac{1}{2}\|\varphi_{n+1}\|_{H}^2 - \frac{1}{2}\|\varphi_{n}\|_{H}^2 
    + \frac{1}{2}\|\varphi_{n+1}-\varphi_{n}\|_{H}^2 
   + i(\varphi_{n+1}) - i(\varphi_{n}) 
\notag \\ 
&+\frac{1}{2}\|\theta_{n+1}\|_{H}^2 - \frac{1}{2}\|\theta_{n}\|_{H}^2 
    + \frac{1}{2}\|\theta_{n+1}-\theta_{n}\|_{H}^2 
+ h(A_{1}\theta_{n+1}, \theta_{n+1})_{H} 
\notag \\
&\leq \frac{1}{2}h\|f_{n+1}\|_{H}^2 + \frac{3}{2}h\|\theta_{n+1}\|_{H}^2 
        + 2h\|v_{n+1}\|_{H}^2 
        + C_{1}h\|\varphi_{n+1}\|_{V_{2}}^2 + C_{2}h. 
\end{align}
Moreover, we sum \eqref{a6} over $n=0, ..., m-1$ with $1 \leq m \leq N$ 
to obtain the inequality  
\begin{align}\label{a7}
&\frac{1}{2}\|L^{1/2}v_{m}\|_{H}^2 
+ \frac{1}{2}\sum_{n=0}^{m-1}\|L^{1/2}(v_{n+1}-v_{n})\|_{H}^2 
+ h\sum_{n=0}^{m-1}\|B^{1/2}v_{n+1}\|_{H}^2  
+ \frac{1}{2}\langle A_{2}^{*}\varphi_{m}, \varphi_{m} \rangle_{V_{2}^{*}, V_{2}} 
\notag \\
&+ \frac{1}{2}\|\varphi_{m}\|_{H}^2 
+ \frac{1}{2}\sum_{n=0}^{m-1}\langle A_{2}^{*}(\varphi_{n+1}-\varphi_{n}), 
                                               \varphi_{n+1}-\varphi_{n} \rangle_{V_{2}^{*}, V_{2}} 
+ \frac{1}{2}\sum_{n=0}^{m-1}\|\varphi_{n+1}-\varphi_{n}\|_{H}^2   
\notag \\
&+ i(\varphi_{m}) + \frac{1}{2}\|\theta_{m}\|_{H}^2 
+ \frac{1}{2}\sum_{n=0}^{m-1}\|\theta_{n+1}-\theta_{n}\|_{H}^2
+ h\sum_{n=0}^{m-1}(A_{1}\theta_{n+1}, \theta_{n+1})_{H} 
\notag \\ 
&\leq \frac{1}{2}\|L^{1/2}v_{0}\|_{H}^2 
   + \frac{1}{2}\langle A_{2}^{*}\varphi_{0}, \varphi_{0} \rangle_{V_{2}^{*}, V_{2}} 
   + \frac{1}{2}\|\varphi_{0}\|_{H}^2 + i(\varphi_{0}) 
   + \frac{1}{2}\|\theta_{0}\|_{H}^2 
   + \frac{1}{2}h\sum_{n=0}^{m-1}\|f_{n+1}\|_{H}^2 
\notag \\
&\,\quad+ \frac{3}{2}h\sum_{n=0}^{m-1}\|\theta_{n+1}\|_{H}^2 
        + 2h\sum_{n=0}^{m-1}\|v_{n+1}\|_{H}^2 
        + C_{1}h\sum_{n=0}^{m-1}\|\varphi_{n+1}\|_{V_{2}}^2 + C_{2}T. 
\end{align}
Here, owing to (C11), it holds that 
\begin{align}\label{a8}
\frac{1}{2}\langle A_{2}^{*}\varphi_{m}, \varphi_{m} \rangle_{V_{2}^{*}, V_{2}} 
+ \frac{1}{2}\|\varphi_{m}\|_{H}^2 
\geq \frac{\omega_{1}}{2}\|\varphi_{m}\|_{V_{2}}^2 
\end{align}
and 
\begin{align}\label{a9}
&\frac{1}{2}\sum_{n=0}^{m-1}\langle A_{2}^{*}(\varphi_{n+1}-\varphi_{n}), 
                                              \varphi_{n+1}-\varphi_{n} \rangle_{V_{2}^{*}, V_{2}} 
+ \frac{1}{2}\sum_{n=0}^{m-1}\|\varphi_{n+1}-\varphi_{n}\|_{H}^2   
\notag \\ 
&\geq \frac{\omega_{1}}{2}\sum_{n=0}^{m-1}\|\varphi_{n+1}-\varphi_{n}\|_{V_{2}}^2 
= \frac{\omega_{1}}{2}h^2\sum_{n=0}^{m-1}\|v_{n+1}\|_{V_{2}}^2.  
\end{align}
Also, we see from (C4) that 
\begin{align}\label{a10}
h\sum_{n=0}^{m-1}(A_{1}\theta_{n+1}, \theta_{n+1})_{H}  
&= h\sum_{n=0}^{m-1}\langle A_{1}^{*}\theta_{n+1}, 
                                                      \theta_{n+1} \rangle_{V_{1}^{*}, V_{1}} 
\notag \\
&\geq \sigma_{1}h\sum_{n=0}^{m-1}\|\theta_{n+1}\|_{V_{1}}^2 
         - h\sum_{n=0}^{m-1}\|\theta_{n+1}\|_{H}^2.  
\end{align}
Hence it follows from \eqref{a7}-\eqref{a10} and (C3) that 
\begin{align*}
&\left(\frac{c_{L}}{2}-2h\right)\|v_{m}\|_{H}^2 
+ \frac{c_{L}}{2}h^2\sum_{n=0}^{m-1}\|z_{n+1}\|_{H}^2 
+ h\sum_{n=0}^{m-1}\|B^{1/2}v_{n+1}\|_{H}^2  
+ \left(\frac{\omega_{1}}{2} - C_{1}h \right)\|\varphi_{m}\|_{V_{2}}^2 
\notag \\
&+ \frac{\omega_{1}}{2}h^2\sum_{n=0}^{m-1}\|v_{n+1}\|_{V_{2}}^2   
+ \frac{1}{2}(1-5h)\|\theta_{m}\|_{H}^2 
+ \frac{1}{2}h^2\sum_{n=0}^{m-1}\|\delta_{h}\theta_{n}\|_{H}^2 
+ \sigma_{1}h\sum_{n=0}^{m-1}\|\theta_{n+1}\|_{V_{1}}^2 
\notag \\
&\leq \frac{1}{2}\|L^{1/2}v_{0}\|_{H}^2 
   + \frac{1}{2}\langle A_{2}^{*}\varphi_{0}, \varphi_{0} \rangle_{V_{2}^{*}, V_{2}} 
   + \frac{1}{2}\|\varphi_{0}\|_{H}^2 + i(\varphi_{0}) 
   + \frac{1}{2}\|\theta_{0}\|_{H}^2 
   + \frac{1}{2}h\sum_{n=0}^{m-1}\|f_{n+1}\|_{H}^2 
\notag \\
&\,\quad+ \frac{5}{2}h\sum_{j=0}^{m-1}\|\theta_{j}\|_{H}^2 
        + 2h\sum_{j=0}^{m-1}\|v_{j}\|_{H}^2 
        + C_{1}h\sum_{j=0}^{m-1}\|\varphi_{j}\|_{V_{2}}^2 + C_{2}T
\end{align*}
and then there exist constants 
$h_{2} \in (0, h_{0})$ and $C_{3} = C_{3}(T) > 0$ such that 
\begin{align}\label{a12}
&\|v_{m}\|_{H}^2 
+ h^2\sum_{n=0}^{m-1}\|z_{n+1}\|_{H}^2 
+ h\sum_{n=0}^{m-1}\|B^{1/2}v_{n+1}\|_{H}^2  
\notag \\ 
&+ \|\varphi_{m}\|_{V_{2}}^2 
+ h^2\sum_{n=0}^{m-1}\|v_{n+1}\|_{V_{2}}^2   
+ \|\theta_{m}\|_{H}^2 
+ h^2\sum_{n=0}^{m-1}\|\delta_{h}\theta_{n}\|_{H}^2
+ h\sum_{n=0}^{m-1}\|\theta_{n+1}\|_{V_{1}}^2 
\notag \\
&\leq  C_{3} h\sum_{j=0}^{m-1}\|\theta_{j}\|_{H}^2 
        + C_{3}h\sum_{j=0}^{m-1}\|v_{j}\|_{H}^2 
        + C_{3}h\sum_{j=0}^{m-1}\|\varphi_{j}\|_{V_{2}}^2 + C_{3}
\end{align}
for all $h \in (0, h_{2})$. 
Therefore the inequality \eqref{a12} and 
the discrete Gronwall lemma (see e.g., \cite[Prop.\ 2.2.1]{Jerome}) imply that 
there exists a constant $C_{4} = C_{4}(T) > 0$ such that 
\begin{align*}
&\|v_{m}\|_{H}^2 
+ h^2\sum_{n=0}^{m-1}\|z_{n+1}\|_{H}^2 
+ h\sum_{n=0}^{m-1}\|B^{1/2}v_{n+1}\|_{H}^2  
\notag \\ 
&+ \|\varphi_{m}\|_{V_{2}}^2 
+ h^2\sum_{n=0}^{m-1}\|v_{n+1}\|_{V_{2}}^2   
+ \|\theta_{m}\|_{H}^2 
+ h^2\sum_{n=0}^{m-1}\|\delta_{h}\theta_{n}\|_{H}^2 
+ h\sum_{n=0}^{m-1}\|\theta_{n+1}\|_{V_{1}}^2 
\leq  C_{4}
\end{align*}
for all $h \in (0, h_{2})$ and $m=1,..., N$.

\end{proof}

\begin{lem}\label{lemkuri2}
Let $h_{2}$ be as in Lemma \ref{lemkuri1}. 
Then there exists a constant $C=C(T)>0$ such that  
\begin{align*}
\Bigl\|\dfrac{d\widehat{\theta}_{h}}{dt}\Bigr\|_{L^2(0, T; H)}^2 
+ h\Bigl\|\dfrac{d\widehat{\theta}_{h}}{dt}\Bigr\|_{L^2(0, T; V_{1})}^2 
+ \|\overline{\theta}_{h}\|_{L^{\infty}(0, T; V_{1})}^2 
\leq C 
\end{align*}
for all $h \in (0, h_{2})$. 
\end{lem}
\begin{proof}
Multiplying the first equation in \ref{Pn} by $\theta_{n+1}-\theta_{n}$ 
and using the Young inequality mean that 
\begin{align}\label{hoho1}
&h\left\|\frac{\theta_{n+1}-\theta_{n}}{h}\right\|_{H}^2 
+ \langle A_{1}^{*}\theta_{n+1}, \theta_{n+1}-\theta_{n} \rangle_{V_{1}^{*}, V_{1}} 
= h\left(f_{n+1} - v_{n+1}, \frac{\theta_{n+1}-\theta_{n}}{h}\right)_{H} 
\notag \\ 
&\leq h\|f_{n+1}\|_{H}^2 + h\|v_{n+1}\|_{H}^2  
+ \frac{1}{2}h\left\|\frac{\theta_{n+1}-\theta_{n}}{h}\right\|_{H}^2.  
\end{align}
Here we derive that 
\begin{align}\label{hoho2}
\langle A_{1}^{*}\theta_{n+1}, \theta_{n+1}-\theta_{n} \rangle_{V_{1}^{*}, V_{1}}  
&= \frac{1}{2}\langle A_{1}^{*}\theta_{n+1}, \theta_{n+1} \rangle_{V_{1}^{*}, V_{1}} 
     - \frac{1}{2}\langle A_{1}^{*}\theta_{n}, \theta_{n} \rangle_{V_{1}^{*}, V_{1}} 
\notag \\ 
&\,\quad+ \frac{1}{2}\langle A_{1}^{*}(\theta_{n+1}-\theta_{n}), 
                                                 \theta_{n+1}-\theta_{n} \rangle_{V_{1}^{*}, V_{1}}. 
\end{align}
Thus, combining \eqref{hoho1} and \eqref{hoho2}, we have 
\begin{align}\label{hoho3}
&\frac{1}{2}h\left\|\frac{\theta_{n+1}-\theta_{n}}{h}\right\|_{H}^2 
+ \frac{1}{2}\langle A_{1}^{*}\theta_{n+1}, \theta_{n+1} \rangle_{V_{1}^{*}, V_{1}} 
     - \frac{1}{2}\langle A_{1}^{*}\theta_{n}, \theta_{n} \rangle_{V_{1}^{*}, V_{1}} 
\notag \\ 
&+ \frac{1}{2}\langle A_{1}^{*}(\theta_{n+1}-\theta_{n}), 
                                                  \theta_{n+1}-\theta_{n} \rangle_{V_{1}^{*}, V_{1}} 
\leq h\|f_{n+1}\|_{H}^2 + h\|v_{n+1}\|_{H}^2.  
\end{align} 
Therefore summing \eqref{hoho3} over $n=0, ..., N-1$ with $1 \leq m \leq N$, 
the condition (C4) and Lemma \ref{lemkuri1} lead to Lemma \ref{lemkuri2}.  
\end{proof}

\begin{lem}\label{lemkuri3}
Let $h_{2}$ be as in Lemma \ref{lemkuri1}. 
Then there exists a constant $C=C(T)>0$ such that  
\begin{align*}
 \|A_{1}\overline{\theta}_{h}\|_{L^2(0, T; H)}^2 
\leq C 
\end{align*}
for all $h \in (0, h_{2})$. 
\end{lem}
\begin{proof}
This lemma holds by the first equation in \ref{Ph}, 
Lemmas \ref{lemkuri1} and \ref{lemkuri2}. 
\end{proof}

\begin{lem}\label{lemkuri4}
Let $h_{2}$ be as in Lemma \ref{lemkuri1}. 
Then there exists a constant $C=C(T)>0$ such that  
\begin{align*}
\|z_{1}\|_{H}^2 + h\|B^{1/2}z_{1}\|_{H}^2 
+ \|v_{1}\|_{V_{2}}^2 + h^2\|z_{1}\|_{V_{2}}^2 
\leq C 
\end{align*}
for all $h\in(0, h_{2})$. 
\end{lem}
\begin{proof}
Thanks to the first equation in \ref{Pn}, 
the identities $v_{1} = v_{0} + hz_{1}$ and $\varphi_{1} = \varphi_{0} + hv_{1}$, 
we can obtain that 
\begin{align}\label{coco1}
Lz_{1} + Bv_{0} + hBz_{1} + A_{2}\varphi_{0} + hA_{2}v_{1} 
+ \Phi\varphi_{1} + {\cal L}\varphi_{1} = \theta_{1}. 
\end{align}
Then, multiplying \eqref{coco1} by $z_{1}$, we can check that 
\begin{align}\label{coco2}
&\|L^{1/2}z_{1}\|_{H}^2 + (Bv_{0}, z_{1})_{H} 
+ h(Bz_{1}, z_{1})_{H} + (A_{2}\varphi_{0}, z_{1})_{H} 
+ h(A_{2}v_{1}, z_{1})_{H}   
\notag \\ 
&+ (\Phi\varphi_{1}, z_{1})_{H} + ({\cal L}\varphi_{1}, z_{1})_{H} 
= (\theta_{1}, z_{1})_{H}. 
\end{align}
Here we see from (C11) that  
\begin{align}\label{coco3}
h(A_{2}v_{1}, z_{1})_{H} 
&= (A_{2}v_{1}, v_{1} - v_{0})_{H} 
= \langle A_{2}^{*}v_{1}, v_{1} - v_{0} \rangle_{V_{2}^{*}, V_{2}}
\notag \\ 
&= \frac{1}{2}\langle A_{2}^{*}v_{1}, v_{1} \rangle_{V_{2}^{*}, V_{2}}
     - \frac{1}{2}\langle A_{2}^{*}v_{0}, v_{0} \rangle_{V_{2}^{*}, V_{2}} 
\notag \\ 
  &\,\quad+ \frac{1}{2}\langle A_{2}^{*}(v_{1} - v_{0}), 
                                                        v_{1} - v_{0} \rangle_{V_{2}^{*}, V_{2}} 
\notag \\ 
&\geq \frac{\omega_{1}}{2}\|v_{1}\|_{V_{2}}^2 - \frac{1}{2}\|v_{1}\|_{H}^2 
     - \frac{1}{2}\langle A_{2}^{*}v_{0}, v_{0} \rangle_{V_{2}^{*}, V_{2}}  
\notag \\ 
&\,\quad + \frac{\omega_{1}}{2}\|v_{1} - v_{0}\|_{V_{2}}^2 
             - \frac{1}{2}\|v_{1} - v_{0}\|_{H}^2.   
\end{align}
The condition (C7) and Lemma \ref{lemkuri1} yield that 
there exists a constant $C_{1} = C_{1}(T) > 0$ satisfying 
\begin{align}\label{coco4}
|(\Phi\varphi_{1}, z_{1})_{H}| 
\leq C_{\Phi}(1 + \|\varphi_{1}\|_{V}^{p})\|\varphi_{1}\|_{V}\|z_{1}\|_{H} 
\leq C_{1}\|z_{1}\|_{H}.  
\end{align}
Thus it follows from \eqref{coco2}-\eqref{coco4} and (C3) that 
\begin{align}\label{coco5}
&c_{L}\|z_{1}\|_{H}^2 + h\|B^{1/2}z_{1}\|_{H}^2 
+ \frac{\omega_{1}}{2}\|v_{1}\|_{V_{2}}^2 + \frac{\omega_{1}}{2}h^2\|z_{1}\|_{V_{2}}^2 
\notag \\ 
&\leq - (Bv_{0}, z_{1})_{H} - (A_{2}\varphi_{0}, z_{1})_{H} 
        + \frac{1}{2}\|v_{1}\|_{H}^2 
        + \frac{1}{2}\langle A_{2}^{*}v_{0}, v_{0} \rangle_{V_{2}^{*}, V_{2}}  
        + \frac{1}{2}\|v_{1} - v_{0}\|_{H}^2  
\notag \\ 
    &\,\quad+ C_{1}\|z_{1}\|_{H} - ({\cal L}\varphi_{1}, z_{1})_{H} + (\theta_{1}, z_{1})_{H}, 
\end{align}
Hence the inequality \eqref{coco5}, the condition (C13), the Young inequality 
and Lemma \ref{lemkuri1} imply that Lemma \ref{lemkuri4} holds. 
\end{proof}

\begin{lem}\label{lemkuri5}
Let $h_{2}$ be as in Lemma \ref{lemkuri1}. 
Then there exist constants $h_{3} \in (0, h_{2})$ and $C=C(T)>0$ such that 
\begin{align*}
\|\overline{z}_{h}\|_{L^{\infty}(0, T; H)}^2  
+ \|B^{1/2}\overline{z}_{h}\|_{L^2(0, T; H)}^2  
+ \|\overline{v}_{h}\|_{L^{\infty}(0, T; V_{2})}^2 
+ h\|\overline{z}_{h}\|_{L^2(0, T; V_{2})}^2  
\leq C 
\end{align*}
for all $h \in (0, h_{3})$. 
\end{lem}
\begin{proof}
Let $n \in \{1, ..., N-1\}$. 
Then the second equation in \ref{Pn} leads to the identity 
\begin{align*}
&L(z_{n+1} - z_{n}) + B(v_{n+1} - v_{n}) + hA_{2}v_{n+1} 
+ \Phi\varphi_{n+1} - \Phi\varphi_{n} 
+ {\cal L}\varphi_{n+1} - {\cal L}\varphi_{n} 
\notag \\ 
&= \theta_{n+1} - \theta_{n}. 
\end{align*} 
Here it holds that 
\begin{align*}
&(L(z_{n+1}-z_{n}), z_{n+1})_{H} 
=(L^{1/2}(z_{n+1}-z_{n}), L^{1/2}z_{n+1})_{H} 
\notag \\ 
&=\frac{1}{2}\|L^{1/2}z_{n+1}\|_{H}^2 - \frac{1}{2}\|L^{1/2}z_{n}\|_{H}^2 
+ \frac{1}{2}\|L^{1/2}(z_{n+1} - z_{n})\|_{H}^2, 
\end{align*}
and hence we have 
\begin{align}\label{pasta1}
&\frac{1}{2}\|L^{1/2}z_{n+1}\|_{H}^2 - \frac{1}{2}\|L^{1/2}z_{n}\|_{H}^2 
+ \frac{1}{2}\|L^{1/2}(z_{n+1} - z_{n})\|_{H}^2 
+ h\|B^{1/2}z_{n+1}\|_{H}^2 
\notag \\ 
&+ \langle A_{2}^{*}v_{n+1}, v_{n+1} - v_{n} \rangle_{V_{2}^{*}, V_{2}} 
  + (v_{n+1}, v_{n+1} - v_{n})_{H} 
\notag \\ 
&= -h\left(\frac{\Phi\varphi_{n+1} - \Phi\varphi_{n}}{h}, z_{n+1} \right)_{H} 
    - h\left(\frac{{\cal L}\varphi_{n+1} - {\cal L}\varphi_{n}}{h}, z_{n+1} \right)_{H} 
\notag \\ 
&\,\quad + h\left(\frac{\theta_{n+1} - \theta_{n}}{h}, z_{n+1} \right)_{H} 
             + h(v_{n+1}, z_{n+1})_{H}. 
\end{align}
On the other hand, we derive that 
\begin{align}\label{pasta2}
&\langle A_{2}^{*}v_{n+1}, v_{n+1} - v_{n} \rangle_{V_{2}^{*}, V_{2}} 
  + (v_{n+1}, v_{n+1} - v_{n})_{H} 
\notag \\ 
&= \frac{1}{2}\langle A_{2}^{*}v_{n+1}, v_{n+1} \rangle_{V_{2}^{*}, V_{2}} 
     - \frac{1}{2}\langle A_{2}^{*}v_{n}, v_{n} \rangle_{V_{2}^{*}, V_{2}} 
     + \frac{1}{2}\langle A_{2}^{*}(v_{n+1} - v_{n}), v_{n+1} - v_{n} \rangle_{V_{2}^{*}, V_{2}} 
\notag \\ 
&\,\quad + \frac{1}{2}\|v_{n+1}\|_{H}^2 - \frac{1}{2}\|v_{n}\|_{H}^2 
             + \frac{1}{2}\|v_{n+1}-v_{n}\|_{H}^2. 
\end{align}
We see from (C7) and Lemma \ref{lemkuri1} that 
there exists a constant $C_{1} = C_{1}(T) > 0$ such that 
\begin{align}\label{pasta3}
-h\left(\frac{\Phi\varphi_{n+1} - \Phi\varphi_{n}}{h}, z_{n+1} \right)_{H} 
&\leq C_{\Phi}h(1 + \|\varphi_{n+1}\|_{V}^{p} + \|\varphi_{n}\|_{V}^{q})
                                                                          \|v_{n+1}\|_{V}\|z_{n+1}\|_{H} 
\notag \\ 
&\leq C_{1}h\|v_{n+1}\|_{V}\|z_{n+1}\|_{H} 
\end{align}
for all $h \in (0, h_{2})$. 
Thus we combine \eqref{pasta1}-\eqref{pasta3} and (C13) to infer that 
there exists a constant $C_{2} = C_{2}(T) > 0$ satisfying 
\begin{align}\label{pasta4}
&\frac{1}{2}\|L^{1/2}z_{n+1}\|_{H}^2 - \frac{1}{2}\|L^{1/2}z_{n}\|_{H}^2 
+ \frac{1}{2}\|L^{1/2}(z_{n+1} - z_{n})\|_{H}^2 
+ h\|B^{1/2}z_{n+1}\|_{H}^2 
\notag \\ 
&+ \frac{1}{2}\langle A_{2}^{*}v_{n+1}, v_{n+1} \rangle_{V_{2}^{*}, V_{2}} 
     - \frac{1}{2}\langle A_{2}^{*}v_{n}, v_{n} \rangle_{V_{2}^{*}, V_{2}} 
     + \frac{1}{2}\langle A_{2}^{*}(v_{n+1} - v_{n}), v_{n+1} - v_{n} \rangle_{V_{2}^{*}, V_{2}} 
\notag \\ 
&+ \frac{1}{2}\|v_{n+1}\|_{H}^2 - \frac{1}{2}\|v_{n}\|_{H}^2 
             + \frac{1}{2}\|v_{n+1}-v_{n}\|_{H}^2 
\notag \\ 
&\leq C_{2}h\|v_{n+1}\|_{V_{2}}\|z_{n+1}\|_{H} 
+ h\left\|\frac{\theta_{n+1} - \theta_{n}}{h}\right\|_{H}\|z_{n+1}\|_{H}
\end{align}
for all $h \in (0, h_{2})$. 
Then summing \eqref{pasta4} over $n=1, ..., \ell-1$ with $2 \leq \ell \leq N$ 
means that 
\begin{align*}
&\frac{1}{2}\|L^{1/2}z_{\ell}\|_{H}^2 
+ \frac{1}{2}\sum_{n=1}^{\ell-1}\|L^{1/2}(z_{n+1} - z_{n})\|_{H}^2 
+ h\sum_{n=1}^{\ell-1}\|B^{1/2}z_{n+1}\|_{H}^2 
\notag \\ 
&+ \frac{1}{2}\langle A_{2}^{*}v_{\ell}, v_{\ell} \rangle_{V_{2}^{*}, V_{2}} 
     + \frac{1}{2}\sum_{n=1}^{\ell-1}
                 \langle A_{2}^{*}(v_{n+1} - v_{n}), v_{n+1} - v_{n} \rangle_{V_{2}^{*}, V_{2}} 
\notag \\ 
&+ \frac{1}{2}\|v_{\ell}\|_{H}^2  
             + \frac{1}{2}\sum_{n=1}^{\ell-1}\|v_{n+1}-v_{n}\|_{H}^2 
\notag \\ 
&\leq \frac{1}{2}\|L^{1/2}z_{1}\|_{H}^2 
+ \frac{1}{2}\langle A_{2}^{*}v_{1}, v_{1} \rangle_{V_{2}^{*}, V_{2}} 
+ \frac{1}{2}\|v_{1}\|_{H}^2  
\notag \\
&\,\quad+ C_{2}h\sum_{n=0}^{\ell-1}\|v_{n+1}\|_{V_{2}}\|z_{n+1}\|_{H} 
+ h\sum_{n=0}^{\ell-1}
                  \left\|\frac{\theta_{n+1} - \theta_{n}}{h}\right\|_{H}\|z_{n+1}\|_{H}, 
\end{align*}
whence it follows from (C3) and (C11) that 
\begin{align}\label{pasta6}
&\frac{c_{L}}{2}\|z_{\ell}\|_{H}^2 
+ h\sum_{n=1}^{\ell-1}\|B^{1/2}z_{n+1}\|_{H}^2 
+ \frac{\omega_{1}}{2}\|v_{\ell}\|_{V_{2}}^2  
             + \frac{\omega_{1}}{2}h^2\sum_{n=1}^{\ell-1}\|z_{n+1}\|_{V_{2}}^2 
\notag \\ 
&\leq \frac{1}{2}\|L^{1/2}z_{1}\|_{H}^2 
+ \frac{1}{2}\langle A_{2}^{*}v_{1}, v_{1} \rangle_{V_{2}^{*}, V_{2}} 
+ \frac{1}{2}\|v_{1}\|_{H}^2  
\notag \\
&\,\quad+ C_{2}h\sum_{n=0}^{\ell-1}\|v_{n+1}\|_{V_{2}}\|z_{n+1}\|_{H} 
+ h\sum_{n=0}^{\ell-1}
                  \left\|\frac{\theta_{n+1} - \theta_{n}}{h}\right\|_{H}\|z_{n+1}\|_{H} 
\end{align}
for all $h \in (0, h_{2})$ and $\ell = 2, ..., N$. 
Therefore we see from \eqref{pasta6}, 
the boundedness of $L$ and $A_{2}^{*}$, 
and Lemma \ref{lemkuri4} that 
there exists a constant $C_{3} = C_{3}(T) > 0$ such that 
\begin{align}\label{pasta7}
&\frac{c_{L}}{2}\|z_{m}\|_{H}^2  
+ h\sum_{n=0}^{m-1}\|B^{1/2}z_{n+1}\|_{H}^2 
+ \frac{\omega_{1}}{2}\|v_{m}\|_{V_{2}}^2  
             + \frac{\omega_{1}}{2}h^2\sum_{n=0}^{m-1}\|z_{n+1}\|_{V_{2}}^2 
\notag \\ 
&\leq C_{3} + C_{2}h\sum_{n=0}^{m-1}\|v_{n+1}\|_{V_{2}}\|z_{n+1}\|_{H} 
+ h\sum_{n=0}^{m-1}
                  \left\|\frac{\theta_{n+1} - \theta_{n}}{h}\right\|_{H}\|z_{n+1}\|_{H}
\end{align}
for all $h \in (0, h_{2})$ and $m=1, ..., N$. 
Moreover, the inequality \eqref{pasta7}, the Young inequality 
and Lemma \ref{lemkuri2} yield that 
there exists a constant $C_{4} = C_{4}(T) > 0$ such that 
\begin{align}\label{pasta8}
&\frac{1}{2}(c_{L} - C_{2}h - h)\|z_{m}\|_{H}^2  
+ h\sum_{n=0}^{m-1}\|B^{1/2}z_{n+1}\|_{H}^2 
+ \frac{1}{2}(\omega_{1} - C_{2}h)\|v_{m}\|_{V_{2}}^2  
\notag \\ 
&+ \frac{\omega_{1}}{2}h^2\sum_{n=0}^{m-1}\|z_{n+1}\|_{V_{2}}^2 
\leq C_{4} + \frac{C_{2}}{2}h\sum_{j=0}^{m-1}\|v_{j}\|_{V_{2}}^2  
+ \frac{1+C_{2}}{2}h\sum_{j=0}^{m-1}\|z_{j}\|_{H}^2  
\end{align}
for all $h \in (0, h_{2})$ and $m=1, ..., N$. 
Thus there exist constants $h_{3} \in (0, h_{2})$ 
and  $C_{5} = C_{5}(T) > 0$ 
such that 
\begin{align*}
&\|z_{m}\|_{H}^2 + h\sum_{n=0}^{m-1}\|B^{1/2}z_{n+1}\|_{H}^2 
+ \|v_{m}\|_{V_{2}}^2 + h^2\sum_{n=0}^{m-1}\|z_{n+1}\|_{V_{2}}^2 
\notag \\ 
&\leq C_{5} + C_{5}h\sum_{j=0}^{m-1}\|v_{j}\|_{V_{2}}^2  
+ C_{5}h\sum_{j=0}^{m-1}\|z_{j}\|_{H}^2  
\end{align*}
for all $h \in (0, h_{3})$ and $m=1, ..., N$. 
Then we infer from 
the discrete Gronwall lemma (see e.g., \cite[Prop.\ 2.2.1]{Jerome}) that 
there exists a constant $C_{6} = C_{6}(T) > 0$ satisfying 
\begin{align*}
\|z_{m}\|_{H}^2 + h\sum_{n=0}^{m-1}\|B^{1/2}z_{n+1}\|_{H}^2 
+ \|v_{m}\|_{V_{2}}^2 + h^2\sum_{n=0}^{m-1}\|z_{n+1}\|_{V_{2}}^2 
\leq C_{6} 
\end{align*}
for all $h \in (0, h_{3})$ and $m=1, ..., N$.

\end{proof}

\begin{lem}\label{lemkuri6}
Let $h_{2}$ be as in Lemma \ref{lemkuri1}. 
Then there exists a constant $C=C(T)>0$ such that 
\begin{align*}
\|\Phi\overline{\varphi}_{h}\|_{L^{\infty}(0, T; H)}  
\leq C 
\end{align*}
for all $h \in (0, h_{2})$. 
\end{lem}
\begin{proof}
This lemma can be proved by (C7) and Lemma \ref{lemkuri1}. 
\end{proof}

\begin{lem}\label{lemkuri7}
Let $h_{3}$ be as in Lemma \ref{lemkuri5}. 
Then there exists a constant $C=C(T)>0$ such that 
\begin{align*}
\|B\overline{v}_{h}\|_{L^2(0, T; H)}^2 
+ \|A_{2}\overline{\varphi}_{h}\|_{L^2(0, T; H)}^2 
\leq C 
\end{align*}
for all $h \in (0, h_{3})$.
\end{lem}
\begin{proof}
We derive from the second equation in \ref{Pn} that 
\begin{align*}
&h\|Bv_{n+1}\|_{H}^2 
= h(Bv_{n+1}, Bv_{n+1})_{H} 
\notag \\ 
&=-h(Lz_{n+1}, Bv_{n+1})_{H} 
    - h(A_{2}\varphi_{n+1}, Bv_{n+1})_{H} 
    - h(\Phi\varphi_{n+1}, Bv_{n+1})_{H} 
\notag \\ 
&\,\quad- h({\cal L}\varphi_{n+1}, Bv_{n+1})_{H} 
    + h(\theta_{n+1}, Bv_{n+1})_{H},   
\end{align*}
and hence it follows from the Young inequality, 
the boundedness of $L$ and (C13) that 
there exists a constant $C_{1} > 0$ satisfying 
\begin{align}\label{pen1}
&h\|Bv_{n+1}\|_{H}^2 
\notag \\ 
&\leq C_{1}h\|z_{n+1}\|_{H}^2 - h(A_{2}\varphi_{n+1}, Bv_{n+1})_{H}  
       + C_{1}h\|\Phi\varphi_{n+1}\|_{H}^2 
       + C_{1}h\|\theta_{n+1}\|_{H}^2  
\end{align}
for all $h \in (0, h_{3})$. 
Here the condition (C6) implies that 
\begin{align}\label{pen2}
&-h(A_{2}\varphi_{n+1}, Bv_{n+1})_{H} 
\notag \\ 
&= -(A_{2}\varphi_{n+1}, B\varphi_{n+1}-B\varphi_{n})_{H} 
\notag \\ 
&= - \frac{1}{2}(A_{2}\varphi_{n+1}, B\varphi_{n+1})_{H} 
     + \frac{1}{2}(A_{2}\varphi_{n}, B\varphi_{n})_{H} 
\notag \\ 
&\,\quad- \frac{1}{2}
                      (A_{2}(\varphi_{n+1}-\varphi_{n}), B(\varphi_{n+1}-\varphi_{n}))_{H}.
\end{align}
Thus, summing \eqref{pen1} over $n=0, ..., m-1$ with $1 \leq m \leq N$, 
we deduce from \eqref{pen2}, 
Lemmas \ref{lemkuri1}, \ref{lemkuri5} and \ref{lemkuri6} that 
there exists a constant $C_{2} = C_{2}(T) > 0$ such that 
\begin{align}\label{pen3}
\|B\overline{v}_{h}\|_{L^2(0, T; H)}^2 \leq C_{2}
\end{align}
for all $h \in (0, h_{3})$. 
Moreover, we see from the second equation in \ref{Ph}, \eqref{pen3}, 
Lemmas \ref{lemkuri1}, \ref{lemkuri5} and \ref{lemkuri6} that 
there exists a constant $C_{3} = C_{3}(T) > 0$ satisfying  
\begin{align*}
\|A_{2}\overline{\varphi}_{h}\|_{L^2(0, T; H)}^2 \leq C_{3}
\end{align*}
for all $h \in (0, h_{3})$. 
\end{proof}

\begin{lem}\label{lemkuri8}
Let $h_{3}$ be as in Lemma \ref{lemkuri5}. 
Then there exists a constant $C=C(T)>0$ such that 
\begin{align*}
&\|\widehat{\varphi}_{h}\|_{W^{1, \infty}(0, T; V_{2})}  
+ \|\widehat{v}_{h}\|_{W^{1, \infty}(0, T; H)}  
\notag \\ 
&+ \|\widehat{v}_{h}\|_{L^{\infty}(0, T; V_{2})}  
+ \|\widehat{\theta}_{h}\|_{H^1(0, T; H)}  
+ \|\widehat{\theta}_{h}\|_{L^{\infty}(0, T; V_{1})}   
\leq C 
\end{align*}
for all $h \in (0, h_{3})$. 
\end{lem}
\begin{proof}
Thanks to 
\eqref{rem1}-\eqref{rem3}, Lemmas \ref{lemkuri1}, \ref{lemkuri2} and \ref{lemkuri5}, 
we can obtain Lemma \ref{lemkuri8}. 
\end{proof}

\begin{prth1.2}
Owing to Lemmas \ref{lemkuri1}-\ref{lemkuri3}, \ref{lemkuri5}-\ref{lemkuri8}, 
and \eqref{rem4}-\eqref{rem6}, 
there exist some functions 
\begin{align*}
&\theta \in H^1(0, T; H) \cap L^{\infty}(0, T; V_{1}) \cap L^2(0, T; D(A_{1})), 
\\ 
&\varphi \in L^{\infty}(0, T; V_{2}) \cap L^2(0, T; D(A_{2})), 
\\ 
&\xi \in L^2(0, T; H)  
\end{align*} 
such that 
\begin{align*}
\frac{d\varphi}{dt} \in L^{\infty}(0, T; V_{2}) \cap L^2(0, T; D(B)),\ 
\frac{d^2\varphi}{dt^2} \in L^{\infty}(0, T; H)  
\end{align*}
and 
\begin{align}
\label{weak1} 
&\widehat{\varphi}_{h} \to \varphi 
\quad \mbox{weakly$^{*}$ in}\ W^{1, \infty}(0, T; V_{2}),  
\\[2.5mm] 
\notag 
&\overline{v}_{h} \to \frac{d\varphi}{dt}     
\quad \mbox{weakly$^{*}$ in}\ L^{\infty}(0, T; V_{2}),  
\\[2.5mm] 
\label{weak2}
&\widehat{v}_{h} \to \frac{d\varphi}{dt} 
\quad \mbox{weakly$^{*}$ in}\ W^{1, \infty}(0, T; H) \cap L^{\infty}(0, T; V_{2}),  
\\[2.5mm] 
\notag 
&\overline{z}_{h} \to \frac{d^2\varphi}{dt^2}   
\quad \mbox{weakly$^{*}$ in}\ L^{\infty}(0, T; H),  
\\[2.5mm] 
\label{weak3}
&L\overline{z}_{h} \to L\frac{d^2\varphi}{dt^2}   
\quad \mbox{weakly$^{*}$ in}\ L^{\infty}(0, T; H),  
\\[2.5mm] 
\label{weak4}
&\widehat{\theta}_{h} \to \theta  
\quad \mbox{weakly$^{*}$ in}\ H^1(0, T; H) \cap L^{\infty}(0, T; V_{1}),   
\\[2.5mm] 
\notag 
&\overline{\varphi}_{h} \to \varphi    
\quad \mbox{weakly$^{*}$ in}\ L^{\infty}(0, T; V_{2}),  
\\[2.5mm] 
\notag 
&\overline{\theta}_{h} \to \theta     
\quad \mbox{weakly$^{*}$ in}\ L^{\infty}(0, T; V_{1}),  
\\[2.5mm] 
\label{weak5}
&A_{1}\overline{\theta}_{h} \to A_{1}\theta     
\quad \mbox{weakly in}\ L^2(0, T; H),  
\\[2.5mm]  
\label{weak6}
&B\overline{v}_{h} \to B\frac{d\varphi}{dt}      
\quad \mbox{weakly in}\ L^2(0, T; H),  
\\[2.5mm] 
\label{weak7}
&A_{2}\overline{\varphi}_{h} \to A_{2}\varphi      
\quad \mbox{weakly in}\ L^2(0, T; H),  
\\[2.5mm] 
\label{weak8}
&\Phi\overline{\varphi}_{h} \to \xi      
\quad \mbox{weakly$^{*}$ in}\ L^{\infty}(0, T; H)
\end{align}
as $h=h_{j}\to+0$. 
Here, since Lemma \ref{lemkuri8}, 
the compactness of the embedding $V_{2} \hookrightarrow H$ and 
the convergence \eqref{weak1} 
yield that 
\begin{align}\label{stronghatvarphi}
\widehat{\varphi}_{h} \to \varphi  
\quad \mbox{strongly in}\ C([0, T]; H)
\end{align}
as $h=h_{j}\to+0$ (see e.g., \cite[Section 8, Corollary 4]{Simon}), 
we infer from \eqref{rem4} and Lemma \ref{lemkuri5} that 
\begin{align}\label{strongoverlinevarphi}
\overline{\varphi}_{h} \to \varphi  
\quad \mbox{strongly in}\ L^{\infty}(0, T; H)
\end{align}
as $h=h_{j}\to+0$. 
Thus it follows from \eqref{weak8} and \eqref{strongoverlinevarphi} that 
\begin{align*}
\int_{0}^{T}(\Phi\overline{\varphi}_{h}(t), \overline{\varphi}_{h}(t))_{H}\,dt 
\to \int_{0}^{T}(\xi(t), \varphi(t))_{H}\,dt 
\end{align*}
as $h=h_{j}\to+0$, whence we have 
\begin{align}\label{xiPhivarphi}
\xi = \Phi\varphi  \quad\mbox{in}\ H\ \mbox{a.e.\ on}\ (0, T)  
\end{align}
(see e.g., \cite[Lemma 1.3, p.\ 42]{Barbu1}).  
On the other hand, we derive from Lemma \ref{lemkuri8}, 
the compactness of the embedding $V_{1} \hookrightarrow H$ and \eqref{weak4} 
that 
\begin{align}\label{stronghattheta}
\widehat{\theta}_{h} \to \theta  
\quad \mbox{strongly in}\ C([0, T]; H)
\end{align}
as $h=h_{j}\to+0$. 
Similarly, we see from \eqref{weak2} that 
\begin{align}\label{stronghatv}
\widehat{v}_{h} \to \frac{d\varphi}{dt}  
\quad \mbox{strongly in}\ C([0, T]; H)
\end{align}
as $h=h_{j}\to+0$. 
Therefore 
we can conclude that there exists a solution of \ref{P} 
by combining \eqref{weak1}, \eqref{weak3}-\eqref{stronghatv}, 
(C13) and by observing that $\overline{f}_{h} \to f$ 
strongly in $L^2(0, T; H)$ as $h\to+0$ (see \cite[Section 5]{CK}).    

Next we establish uniqueness of solutions to \ref{P}. 
We let $(\theta, \varphi)$, $(\overline{\theta}, \overline{\varphi})$ 
be two solutions of \ref{P} and put 
$\widetilde{\theta}:=\theta-\overline{\theta}$, 
$\widetilde{\varphi}:=\varphi-\overline{\varphi}$. 
Then the identity \eqref{df1} means that  
\begin{align}\label{cocoro1}
\frac{1}{2}\frac{d}{dt}\|\widetilde{\theta}(t)\|_{H}^2 
+ \left(\frac{d\widetilde{\varphi}}{dt}(t), \widetilde{\theta}(t)\right)_{H} 
+ (A_{1}\widetilde{\theta}(t), \widetilde{\theta}(t))_{H} 
= 0.  
\end{align}
Here, by \eqref{df2}, the Young inequality, 
(C7), (C13), Lemma \ref{lemkuri1} 
and the continuity of the embedding $V_{2} \hookrightarrow H$,    
we can verify that there exists a constant $C_{1}=C_{1}(T)>0$ such that 
\begin{align}\label{cocoro2}
&\frac{1}{2}\frac{d}{dt}\left\|L^{1/2}\frac{d\widetilde{\varphi}}{dt}(t)\right\|_{H}^2 
+ \left(B\frac{d\widetilde{\varphi}}{dt}(t), \frac{d\widetilde{\varphi}}{dt}(t)\right)_{H} 
+ \frac{1}{2}\frac{d}{dt}\Bigl\|A_{2}^{1/2}\widetilde{\varphi}(t)\Bigr\|_{H}^2 
\notag \\[1mm] 
&= \left(\widetilde{\theta}(t), \frac{d\widetilde{\varphi}}{dt}(t)\right)_{H} 
     - \left(\Phi\varphi(t)-\Phi\overline{\varphi}(t), 
                                                          \frac{d\widetilde{\varphi}}{dt}(t)\right)_{H} 
      - \left({\cal L}\varphi(t)-{\cal L}\overline{\varphi}(t), 
                                                          \frac{d\widetilde{\varphi}}{dt}(t)\right)_{H} 
\notag \\ 
&\leq \left(\widetilde{\theta}(t), \frac{d\widetilde{\varphi}}{dt}(t)\right)_{H}  
         + \frac{C_{\Phi}^2}{2}(1+\|\varphi(t)\|_{V}^{p} + \|\overline{\varphi}(t)\|_{V}^{q})^2
                                                                               \|\widetilde{\varphi}(t)\|_{V}^2 
\notag \\ 
&\,\quad+ \frac{C_{{\cal L}}^2}{2}\|\widetilde{\varphi}(t)\|_{H}^2 
         + \left\|\frac{d\widetilde{\varphi}}{dt}(t)\right\|_{H}^2 
\notag \\[1.5mm]
&\leq  \left(\widetilde{\theta}(t), \frac{d\widetilde{\varphi}}{dt}(t)\right)_{H} 
         + C_{1}\|\widetilde{\varphi}(t)\|_{V_{2}}^2 
         + \frac{1}{c_{L}}\left\|L^{1/2}\frac{d\widetilde{\varphi}}{dt}(t)\right\|_{H}^2    
\end{align}
for a.a.\ $t\in(0, T)$. 
Also, the Young inequality, (C3) 
and the continuity of the embedding $V_{2} \hookrightarrow H$ 
imply that 
there exists a constant $C_{2}>0$ such that 
\begin{align}\label{cocoro3}
\frac{1}{2}\frac{d}{dt}\|\widetilde{\varphi}(t)\|_{H}^2 
= \left(\frac{d\widetilde{\varphi}}{dt}(t), \widetilde{\varphi}(t)\right)_{H} 
\leq \frac{1}{2c_{L}}\left\|L^{1/2}\frac{d\widetilde{\varphi}}{dt}(t)\right\|_{H}^2 
       + C_{2}\|\widetilde{\varphi}(t)\|_{V_{2}}^2    
\end{align}
for a.a.\ $t \in (0, T)$. 
Hence we deduce from \eqref{cocoro1}-\eqref{cocoro3}, the integration over $(0, t)$, 
where $t \in [0, T]$, \eqref{df3} and the monotonicity of $A_{1}$, $B$ that 
there exists a constant $C_{3}=C_{3}(T)>0$ such that 
\begin{align}\label{cocoro4} 
&\frac{1}{2}\|\widetilde{\theta}(t)\|_{H}^2 
+ \frac{1}{2}\left\|L^{1/2}\frac{d\widetilde{\varphi}}{dt}(t)\right\|_{H}^2 
+ \frac{1}{2}\Bigl\|A_{2}^{1/2}\widetilde{\varphi}(t)\Bigr\|_{H}^2 
+ \frac{1}{2}\|\widetilde{\varphi}\|_{H}^2 
\notag \\ 
&\leq C_{3}\int_{0}^{t}\left\|L^{1/2}\frac{d\widetilde{\varphi}}{dt}(s)\right\|_{H}^2\,ds 
         + C_{3}\int_{0}^{t}\|\widetilde{\varphi}(s)\|_{V_{2}}^2\,ds      
\end{align}
for all $t \in [0, T]$. 
Here, owing to (C11), it holds that 
\begin{align}\label{cocoro5}
&\frac{1}{2}\Bigl\|A_{2}^{1/2}\widetilde{\varphi}(t)\Bigr\|_{H}^2 
+ \frac{1}{2}\|\widetilde{\varphi}\|_{H}^2 
\notag \\ 
&= \frac{1}{2}\langle A_{2}^{*}\widetilde{\varphi}(t), 
                                            \widetilde{\varphi}(t) \rangle_{V_{2}^{*}, V_{2}} 
    + \frac{1}{2}\|\widetilde{\varphi}\|_{H}^2 
\geq \frac{\omega_{1}}{2}\|\widetilde{\varphi}(t)\|_{V_{2}}^2. 
\end{align}
Thus it follows from \eqref{cocoro4} and \eqref{cocoro5} that 
\begin{align*}
&\frac{1}{2}\|\widetilde{\theta}(t)\|_{H}^2 
+ \frac{1}{2}\left\|L^{1/2}\frac{d\widetilde{\varphi}}{dt}(t)\right\|_{H}^2 
+ \frac{\omega_{1}}{2}\|\widetilde{\varphi}(t)\|_{V_{2}}^2 
\notag \\ 
&\leq C_{3}\int_{0}^{t}\left\|L^{1/2}\frac{d\widetilde{\varphi}}{dt}(s)\right\|_{H}^2\,ds 
         + C_{3}\int_{0}^{t}\|\widetilde{\varphi}(s)\|_{V_{2}}^2\,ds      
\end{align*} 
and then applying the Gronwall lemma yields that 
$\widetilde{\theta}=\widetilde{\varphi}=0$, 
which leads to the identities 
$\theta=\overline{\theta}$ and $\varphi=\overline{\varphi}$. 
\qed  
\end{prth1.2}

\vspace{10pt}
 
\section{Error estimates}\label{Sec5}

In this section we will prove Theorem \ref{erroresti}. 

\begin{lem}\label{semierroresti}
Let $h_{3}$ be as in Lemma \ref{lemkuri5}. 
Then there exists a constant $C=C(T)>0$ such that 
\begin{align*}
&\|L^{1/2}(\widehat{v}_{h} - v)\|_{L^{\infty}(0, T; H)} 
+ \|B^{1/2}(\overline{v}_{h}-v)\|_{L^2(0, T; H)} 
+ \|\widehat{\varphi}_{h} - \varphi\|_{L^{\infty}(0, T; V_{2})} 
\\ 
&+ \|\widehat{\theta}_{h} - \theta\|_{L^{\infty}(0, T; H)} 
+ \|\overline{\theta}_{h} - \theta\|_{L^2(0, T; V_{1})} 
\leq C h^{1/2} + C\|\overline{f}_{h} - f\|_{L^2(0, T; H)}
\end{align*}
for all $h \in (0, h_{3})$, where $v = \frac{d\varphi}{dt}$. 
\end{lem}
\begin{proof}
We infer from the first equations in \ref{Ph} and \eqref{df1} that 
\begin{align}\label{e1}
&\frac{1}{2}\frac{d}{dt}\|\widehat{\theta}_{h}(t) - \theta(t)\|_{H}^2 
\notag \\ 
&= - (\overline{v}_{h}(t)-v(t), \widehat{\theta}_{h}(t)-\theta(t))_{H} 
    - (A_{1}(\overline{\theta}_{h}(t)-\theta(t)), 
                              \widehat{\theta}_{h}(t) - \overline{\theta}_{h}(t))_{H} 
\notag \\
    &\,\quad- \langle A_{1}^{*}(\overline{\theta}_{h}(t)-\theta(t)), 
                                 \overline{\theta}_{h}(t)-\theta(t) \rangle_{V_{1}^{*}, V_{1}} 
    + (\overline{f}_{h}(t) - f(t), \widehat{\theta}_{h}(t) - \theta(t))_{H}.  
\end{align}
Here we derive from the Young inequality and (C3) that 
\begin{align}\label{e2}
&- (\overline{v}_{h}(t)-v(t), \widehat{\theta}_{h}(t)-\theta(t))_{H} 
\notag \\ 
&\leq \frac{1}{2}\|\overline{v}_{h}(t)-v(t)\|_{H}^2 
       + \frac{1}{2}\|\widehat{\theta}_{h}(t)-\theta(t)\|_{H}^2 
\notag \\ 
&\leq \|\overline{v}_{h}(t)-\widehat{v}_{h}(t)\|_{H}^2 
         + \|\widehat{v}_{h}(t)-v(t)\|_{H}^2 
         + \frac{1}{2}\|\widehat{\theta}_{h}(t)-\theta(t)\|_{H}^2 
\notag \\ 
&\leq \|\overline{v}_{h}(t)-\widehat{v}_{h}(t)\|_{H}^2 
         + \frac{1}{c_{L}}\|L^{1/2}(\widehat{v}_{h}(t)-v(t))\|_{H}^2 
         + \frac{1}{2}\|\widehat{\theta}_{h}(t)-\theta(t)\|_{H}^2.  
\end{align} 
It follows from (C4) that 
\begin{align}\label{e3}
&- \langle A_{1}^{*}(\overline{\theta}_{h}(t)-\theta(t)), 
                                      \overline{\theta}_{h}(t)-\theta(t) \rangle_{V_{1}^{*}, V_{1}} 
\notag \\ 
&\leq -\sigma_{1}\|\overline{\theta}_{h}(t)-\theta(t)\|_{V_{1}}^2 
         + \|\overline{\theta}_{h}(t)-\theta(t)\|_{H}^2  
\notag \\ 
&\leq -\sigma_{1}\|\overline{\theta}_{h}(t)-\theta(t)\|_{V_{1}}^2 
         + 2\|\overline{\theta}_{h}(t) - \widehat{\theta}_{h}(t)\|_{H}^2 
         + 2\|\widehat{\theta}_{h}(t) - \theta(t)\|_{H}^2.  
\end{align} 
We have from the Young inequality that  
\begin{align}\label{e4}
(\overline{f}_{h}(t) - f(t), \widehat{\theta}_{h}(t) - \theta(t))_{H} 
\leq \frac{1}{2}\|\overline{f}_{h}(t) - f(t)\|_{H}^2 
+ \frac{1}{2}\|\widehat{\theta}_{h}(t) - \theta(t)\|_{H}^2.  
\end{align}
Thus we see from \eqref{e1}-\eqref{e4} and the integration over $(0, t)$, 
where $t \in [0, T]$, Lemma \ref{lemkuri3}, \eqref{rem6}, Lemma \ref{lemkuri2}, 
\eqref{rem5} and Lemma \ref{lemkuri5} that 
there exists a constant $C_{1} = C_{1}(T) > 0$ such that 
\begin{align}\label{e5}
&\frac{1}{2}\|\widehat{\theta}_{h}(t) - \theta(t)\|_{H}^2 
+ \sigma_{1}\int_{0}^{t}\|\overline{\theta}_{h}(s) - \theta(s)\|_{V_{1}}^2\,ds 
\notag \\ 
&\leq C_{1}h 
+ C_{1}\int_{0}^{t} \|L^{1/2}(\widehat{v}_{h}(s)-v(s))\|_{H}^2\,ds 
+ C_{1}\int_{0}^{t}\|\widehat{\theta}_{h}(s)-\theta(s)\|_{H}^2\,ds 
\notag \\ 
&\,\quad+ C_{1}\|\overline{f}_{h} - f\|_{L^2(0, T; H)}^2 
\end{align}
for all $t \in [0, T]$ and all $h \in (0, h_{3})$. 

Next we observe that the identity $\frac{d\widehat{v}_{h}}{dt}=\overline{z}_{h}$, 
putting $z:=\frac{dv}{dt}$,  
the second equations in \ref{Ph} and \eqref{df2} imply that 
\begin{align}\label{e6}
&\frac{1}{2}\frac{d}{dt}\|L^{1/2}(\widehat{v}_{h}(t)-v(t))\|_{H}^2 
\notag \\ 
&= (L(\overline{z}_{h}(t)-z(t)), \widehat{v}_{h}(t)-\overline{v}_{h}(t))_{H} 
     + (L(\overline{z}_{h}(t)-z(t)), \overline{v}_{h}(t)-v(t))_{H} 
\notag \\ 
&= (L(\overline{z}_{h}(t)-z(t)), \widehat{v}_{h}(t)-\overline{v}_{h}(t))_{H} 
     - (B(\overline{v}_{h}(t)-v(t)), \overline{v}_{h}(t)-v(t))_{H} 
\notag \\ 
&\,\quad - (A_{2}(\overline{\varphi}_{h}(t)-\varphi(t)), \overline{v}_{h}(t)-v(t))_{H} 
- (\Phi\overline{\varphi}_{h}(t)-\Phi\varphi(t), \overline{v}_{h}(t)-v(t))_{H} 
\notag \\ 
&\,\quad- ({\cal L}\overline{\varphi}_{h}(t)-{\cal L}\varphi(t), 
                                                                          \overline{v}_{h}(t)-v(t))_{H} 
+ (\overline{\theta}_{h}(t) - \theta(t), \overline{v}_{h}(t)-v(t))_{H}.    
\end{align}
Here, recalling that the linear operator $L : H \to H$ is bounded, 
we can obtain that 
there exists a constant $C_{2}>0$ such that 
\begin{align}\label{e7}
(L(\overline{z}_{h}(t)-z(t)), \widehat{v}_{h}(t)-\overline{v}_{h}(t))_{H} 
&\leq \|L(\overline{z}_{h}(t)-z(t))\|_{H}\|\widehat{v}_{h}(t)-\overline{v}_{h}(t)\|_{H} 
\notag \\ 
&\leq C_{2}\|\overline{z}_{h}(t)-z(t)\|_{H}\|\widehat{v}_{h}(t)-\overline{v}_{h}(t)\|_{H} 
\end{align}
for a.a.\ $t \in (0, T)$ and all $h \in (0, h_{3})$. 
Owing to 
the identities $\overline{v}_{h}=\frac{d\widehat{\varphi}_{h}}{dt}$, 
$v=\frac{d\varphi}{dt}$ and    
the boundedness of the operator $A_{2}^{*} : V_{2} \to V_{2}^{*}$, 
it holds that there exists a constant $C_{3}>0$ such that 
\begin{align}\label{e8}
&- (A_{2}(\overline{\varphi}_{h}(t)-\varphi(t)), \overline{v}_{h}(t)-v(t))_{H} 
\notag \\ 
&= - \langle A_{2}^{*}(\overline{\varphi}_{h}(t)-\widehat{\varphi}_{h}(t)), 
                                                    \overline{v}_{h}(t)-v(t) \rangle_{V_{2}^{*}, V_{2}} 
- \frac{1}{2}\frac{d}{dt}\|A_{2}^{1/2}(\widehat{\varphi}_{h}(t)-\varphi(t))\|_{H}^2 
\notag \\ 
&\leq C_{3}\|\overline{\varphi}_{h}(t)-\widehat{\varphi}_{h}(t)\|_{V_{2}}
                                                                   \|\overline{v}_{h}(t)-v(t)\|_{V_{2}} 
- \frac{1}{2}\frac{d}{dt}\|A_{2}^{1/2}(\widehat{\varphi}_{h}(t)-\varphi(t))\|_{H}^2 
\end{align}
for a.a.\ $t \in (0, T)$ and all $h \in (0, h_{3})$. 
We derive from (C7), Lemma \ref{lemkuri1}, the Young inequality and (C3) that 
there exists a constant $C_{4}=C_{4}(T)>0$ such that 
\begin{align}\label{e9}
&- (\Phi\overline{\varphi}_{h}(t)-\Phi\varphi(t), \overline{v}_{h}(t)-v(t))_{H} 
\notag \\ 
&\leq C_{\Phi}(1 + \|\overline{\varphi}_{h}(t)\|_{V}^{p} + \|\varphi(t)\|_{V}^{q})
               \|\overline{\varphi}_{h}(t)-\varphi(t)\|_{V}\|\overline{v}_{h}(t)-v(t)\|_{H} 
\notag \\ 
&\leq C_{4}\|\overline{\varphi}_{h}(t)-\varphi(t)\|_{V}\|\overline{v}_{h}(t)-v(t)\|_{H} 
\notag \\ 
&\leq \frac{C_{4}}{2}\|\overline{\varphi}_{h}(t)-\varphi(t)\|_{V}^2 
         + \frac{C_{4}}{2}\|\overline{v}_{h}(t)-v(t)\|_{H}^2 
\notag \\ 
&\leq C_{4}\|\overline{\varphi}_{h}(t)-\widehat{\varphi}_{h}(t)\|_{V}^2 
         + C_{4}\|\widehat{\varphi}_{h}(t)-\varphi(t)\|_{V}^2 
\notag \\
    &\,\quad+ C_{4}\|\overline{v}_{h}(t)-\widehat{v}_{h}(t)\|_{H}^2 
         + \frac{C_{4}}{c_{L}}\|L^{1/2}(\widehat{v}_{h}(t)-v(t))\|_{H}^2 
\end{align}
for a.a.\ $t \in (0, T)$ and all $h \in (0, h_{3})$. 
It follows from (C13), 
the continuity of the embedding $V \hookrightarrow H$, 
the Young inequality and (C3) that 
there exists a constant $C_{5}>0$ satisfying 
\begin{align}\label{e10}
&- ({\cal L}\overline{\varphi}_{h}(t)-{\cal L}\varphi(t), \overline{v}_{h}(t)-v(t))_{H} 
\notag \\ 
&\leq C_{5}\|\overline{\varphi}_{h}(t)-\varphi(t)\|_{V}\|\overline{v}_{h}(t)-v(t)\|_{H} 
\notag \\ 
&\leq \frac{C_{5}}{2}\|\overline{\varphi}_{h}(t)-\varphi(t)\|_{V}^2 
         + \frac{C_{5}}{2}\|\overline{v}_{h}(t)-v(t)\|_{H}^2 
\notag \\ 
&\leq C_{5}\|\overline{\varphi}_{h}(t)-\widehat{\varphi}_{h}(t)\|_{V}^2 
         + C_{5}\|\widehat{\varphi}_{h}(t)-\varphi(t)\|_{V}^2 
\notag \\
    &\,\quad+ C_{5}\|\overline{v}_{h}(t)-\widehat{v}_{h}(t)\|_{H}^2 
         + \frac{C_{5}}{c_{L}}\|L^{1/2}(\widehat{v}_{h}(t)-v(t))\|_{H}^2.  
\end{align}
The Young inequality and (C3) yield that 
\begin{align}\label{e11}
&(\overline{\theta}_{h}(t) - \theta(t), \overline{v}_{h}(t)-v(t))_{H} 
\notag \\ 
&= (\overline{\theta}_{h}(t) - \widehat{\theta}_{h}(t), \overline{v}_{h}(t)-v(t))_{H} 
    + (\widehat{\theta}_{h}(t) - \theta(t), \overline{v}_{h}(t) - \widehat{v}_{h}(t))_{H} 
\notag \\ 
&\,\quad + (\widehat{\theta}_{h}(t) - \theta(t), \widehat{v}_{h}(t) - v(t))_{H} 
\notag \\ 
&\leq \|\overline{\theta}_{h}(t) - \widehat{\theta}_{h}(t)\|_{H}
                                                                    \|\overline{v}_{h}(t)-v(t)\|_{H} 
        + \|\widehat{\theta}_{h}(t) - \theta(t)\|_{H}
                                      \|\overline{v}_{h}(t) - \widehat{v}_{h}(t)\|_{H} 
\notag \\ 
&\,\quad + \frac{1}{2}\|\widehat{\theta}_{h}(t) - \theta(t)\|_{H}^2 
             + \frac{1}{2c_{L}}\|L^{1/2}(\widehat{v}_{h}(t) - v(t))\|_{H}^2.   
\end{align}
Thus we infer from \eqref{e6}-\eqref{e11}, 
the monotonicity of $\Phi$, 
the integration over $(0, t)$, where $t \in [0, T]$, 
\eqref{rem4}-\eqref{rem6}, Lemmas \ref{lemkuri2} and \ref{lemkuri5}    
that there exists a constant $C_{6}=C_{6}(T)>0$ such that 
\begin{align}\label{e12}
&\frac{1}{2}\|L^{1/2}(\widehat{v}_{h}(t)-v(t))\|_{H}^2 
+ \frac{1}{2}\|A_{2}^{1/2}(\widehat{\varphi}_{h}(t)-\varphi(t))\|_{H}^2 
+ \int_{0}^{t}\|B^{1/2}(\overline{v}_{h}(s)-v(s))\|_{H}^2\,ds  
\notag \\ 
&\leq C_{6}h   
+ C_{6}\int_{0}^{t}\|\widehat{\varphi}_{h}(s)-\varphi(s)\|_{V}^2\,ds  
+ C_{6}\int_{0}^{t}\|L^{1/2}(\widehat{v}_{h}(s)-v(s))\|_{H}^2\,ds  
\notag \\ 
&\,\quad+ C_{6}\int_{0}^{t}\|\widehat{\theta}_{h}(s) - \theta(s)\|_{H}^2\,ds
\end{align}
for all $t \in [0, T]$ and all $h \in (0, h_{3})$. 
On the other hand, 
we have from the identities $\frac{d\widehat{\varphi}_{h}}{dt}=\overline{v}_{h}$, 
$\frac{d\varphi}{dt}=v$, 
the Young inequality, (C3) 
and the continuity of the embedding $V_{2} \hookrightarrow H$ 
that there exists a constant $C_{7}>0$ such that 
\begin{align}\label{e13}
&\frac{1}{2}\frac{d}{dt}\|\widehat{\varphi}_{h}(t)-\varphi(t)\|_{H}^2 
\notag \\ 
&= (\overline{v}_{h}(t)-v(t), \widehat{\varphi}_{h}(t)-\varphi(t))_{H} 
\notag \\ 
&\leq \frac{1}{2}\|\overline{v}_{h}(t)-v(t)\|_{H}^2 
         + \frac{1}{2}\|\widehat{\varphi}_{h}(t)-\varphi(t)\|_{H}^2 
\notag \\ 
&\leq \|\overline{v}_{h}(t)-\widehat{v}_{h}(t)\|_{H}^2 
         + \frac{1}{c_{L}}\|L^{1/2}(\widehat{v}_{h}(t) - v(t))\|_{H}^2  
         + C_{7}\|\widehat{\varphi}_{h}(t)-\varphi(t)\|_{V_{2}}^2 
\end{align}
for a.a.\ $t \in (0, T)$ and all $h \in (0, h_{3})$. 
Hence we derive from \eqref{e12}, the integration \eqref{e13} over $(0, t)$, 
where $t \in [0, T]$, and (C11) 
that there exists a constant $C_{8}=C_{8}(T)>0$ satisfying 
\begin{align}\label{e15}
&\frac{1}{2}\|L^{1/2}(\widehat{v}_{h}(t)-v(t))\|_{H}^2 
+ \frac{\omega_{1}}{2}\|\widehat{\varphi}_{h}(t) - \varphi(t)\|_{V_{2}}^2 
+ \int_{0}^{t}\|B^{1/2}(\overline{v}_{h}(s)-v(s))\|_{H}^2\,ds  
\notag \\ 
&\leq C_{8}h   
+ C_{8}\int_{0}^{t}\|\widehat{\varphi}_{h}(s)-\varphi(s)\|_{V_{2}}^2\,ds  
+ C_{8}\int_{0}^{t}\|L^{1/2}(\widehat{v}_{h}(s)-v(s))\|_{H}^2\,ds  
\notag \\ 
&\,\quad+ C_{8}\int_{0}^{t}\|\widehat{\theta}_{h}(s) - \theta(s)\|_{H}^2\,ds. 
\end{align}    

Therefore combining \eqref{e5} and \eqref{e15} means that 
there exists a constant $C_{9}=C_{9}(T)>0$ such that 
\begin{align*}
&\frac{1}{2}\|L^{1/2}(\widehat{v}_{h}(t)-v(t))\|_{H}^2 
+ \frac{\omega_{1}}{2}\|\widehat{\varphi}_{h}(t) - \varphi(t)\|_{V_{2}}^2 
+ \int_{0}^{t}\|B^{1/2}(\overline{v}_{h}(s)-v(s))\|_{H}^2\,ds  
\notag \\ 
& + \frac{1}{2}\|\widehat{\theta}_{h}(t) - \theta(t)\|_{H}^2 
+ \sigma_{1}\int_{0}^{t}\|\overline{\theta}_{h}(s) - \theta(s)\|_{V_{1}}^2\,ds 
\notag \\ 
&\leq C_{9}h 
+ C_{9}\int_{0}^{t}\|\widehat{\varphi}_{h}(s) - \varphi(s)\|_{V_{2}}^2\,ds 
+ C_{9}\int_{0}^{t}\|L^{1/2}(\widehat{v}_{h}(s)-v(s))\|_{H}^2\,ds  
\notag \\ 
&\,\quad+ C_{9}\|\overline{f}_{h} - f\|_{L^2(0, T; H)}^2 
+ C_{9}\int_{0}^{t}\|\widehat{\theta}_{h}(s) - \theta(s)\|_{H}^2\,ds
\end{align*}
for all $t \in [0, T]$ and all $h \in (0, h_{3})$.  
Then, applying the Gronwall lemma, we can obtain Lemma \ref{semierroresti}.    
\end{proof}

\begin{prth1.3}
Observing that there exists a constant $C_{1}>0$ such that 
$$
\|\overline{f}_{h}-f\|_{L^2(0, T; H)} \leq C_{1}h^{1/2}  
$$
for all $h>0$ (see \cite[Section 5]{CK}),  
we can prove Theorem \ref{erroresti} by Lemma \ref{semierroresti}. 
\qed  
\end{prth1.3}

\section*{Acknowledgments}
The author is supported by JSPS Research Fellowships 
for Young Scientists (No.\ 18J21006).   
%
%
%

\end{document}